\documentclass[notitlepage,11pt,a4paper]{scrartcl}

\usepackage[UKenglish]{babel}
\usepackage{abstract}
\usepackage[utf8]{inputenc}
\usepackage{amsmath}
\usepackage{amsfonts,amssymb,latexsym,exscale}
\usepackage{xspace}
\usepackage{amsthm}
\usepackage{txfonts}
\usepackage{upgreek}
\usepackage{tikz}
\usetikzlibrary{matrix,arrows,decorations.pathreplacing}
\usepackage{tensor} 
\usepackage{stmaryrd} 
\usepackage{mathtools} 
\usepackage{url}
\newtheorem{theo}{Theorem}[section]
	
	\newtheorem{main}[theo]{Main Theorem}
	\newtheorem{cor}[theo]{Corollary}
	\newtheorem*{cor*}{Corollary}	

	\newtheorem{lem}[theo]{Lemma}
	\newtheorem*{lem*}{Lemma}
	
	\newtheorem{prop}[theo]{Proposition}
	\newtheorem*{prop*}{Proposition}
	\newtheorem{df}[theo]{Definition}
	\newtheorem*{df*}{Definition}

	\newtheorem*{Example*}{Example}
 \newtheorem*{rem*}{Remark}
 \newtheorem{rem}{Remark}  
\numberwithin{rem}{theo}

\numberwithin{equation}{section}

\numberwithin{theo}{section}

\makeatletter
\@addtoreset{rem}{section}
\makeatother
\DeclareMathOperator{\ldef}{\mathrel{\mathop:}=}

\DeclareMathOperator{\rdef}{=\mathrel{\mathop:}}

\DeclareMathOperator{\im}{\mathrm{im}} 
\DeclareMathOperator{\coker}{\mathrm{coker}}
\DeclareMathOperator{\NN}{\mathbb{N}} 

\DeclareMathOperator{\CC}{\mathbb{C}}

\DeclareMathOperator{\id}{\mathrm{id}}

\DeclareMathOperator{\B}{\mathcal{B}} 

\DeclareMathOperator{\E}{E}

\DeclareMathOperator{\I}{\mathfrak{I}}

\DeclareMathOperator{\konk}{\centerdot}
\DeclareMathOperator{\ohne}{\!\setminus}
\DeclareMathOperator{\ru}{\leftharpoonup}
\DeclareMathOperator{\kleint}{\prec}
\DeclareMathOperator{\nkleint}{\nprec}
\DeclareMathOperator{\groint}{\succ}
\DeclareMathOperator{\wa}{\relbar\relbar}
\DeclareMathOperator{\stver}{\,\rightrightarrows\,} 
\DeclareMathOperator{\nstver}{\rightrightarrows\!\!\!\!\! \left\vert\right. \;} 
\DeclareMathOperator{\Po}{\mathcal{P}} 
\DeclareMathOperator{\W}{{W}} 
\DeclareMathOperator{\tW}{\widetilde{\W}}
\DeclareMathOperator{\Q}{\widetilde{Q}} 
\DeclareMathOperator{\K}{\widetilde{K}} 

 \newcommand{\case}[1]{\underline{\smash{#1}\rule[-0.15ex]{0pt}{1ex}}\quad}
 \newcommand{\pot}[1]{\ensuremath{\wp(#1)}}
 \newcommand{\mo}[1]{\ensuremath{#1\text{-}\mathrm{mod}}}
 \newcommand{\gen}[1]{\ensuremath{\mathrm{gen}\textit{-}#1}}
 \newcommand{\rep}[2]{\ensuremath{#2\text{-}\tensor*[_{#1}]{\mathrm{rep}}{}}}
 \newcommand{\Inc}[1]{\ensuremath{\mathrm{Inc}\left(#1\right)}}
 \newcommand{\Ext}[3]{\ensuremath{\mathrm{Ext}_{#1}^{#2}\!\left(#3\right)}}
 \newcommand{\Hom}[2]{\ensuremath{\mathrm{Hom}_{#1}\!\left(#2\right)}}
 \newcommand{\End}[2]{\ensuremath{\mathrm{End}_{#1}\!\left(#2\right)}}

 \newcommand{\rad}[1]{\ensuremath{\mathrm{rad}({#1})}}
 \newcommand{\extko}[1]{\ensuremath{\Gamma(#1)}}
 \newcommand{\set}[2]{\ensuremath{\{#1\:|\;#2\}}}
 
 \newcommand{\Ann}[2]{\ensuremath{\mathrm{Ann}_{#1}(#2)}}
 \newcommand{\kong}[1]{\ensuremath{\underset{#1}{\cong}}} 
 \newcommand{\itonew}[1]{\ensuremath{{\underline{#1\,}{\raise-0.4ex\hbox{\vphantom{$#1$ }\rule[-0.114ex]{0.125ex}{1.5ex}}}}\, }}
 \newcommand{\ito}[1]{\ensuremath{\itonew{#1}}}
 \newcommand{\oto}[1]{\ensuremath{\itonew{#1}\cup\{0\}}}
 \newcommand{\A}[1]{\ensuremath{\mathcal{A}_{#1}}} 
 \newcommand{\subf}[1]{\ensuremath{\mathsf{t}_{#1}}} 
 \newcommand{\p}[1]{\ensuremath{\mathsf{P}\!_{#1}}} 
 
 \newcommand{\monoid}[1]{\ensuremath{\uppi_{#1}}} 
 \newcommand{\free}[1]{\ensuremath{\tensor*{#1}{*^{*}_{0}}}} 
 \newcommand{\senke}[2]{\ensuremath{\mathrm{S}_{#2}{#1}}} 
 \newcommand{\y}[2]{\ensuremath{\tensor*{y}{*^{(#2)}_{#1}}}}
 \newcommand{\yJ}[1]{\ensuremath{y(#1)}} 
 \newcommand{\Jk}[2]{\ensuremath{\tensor*{\I}{*^{(#2)}_{#1}}}}
 \newcommand{\z}[2]{\ensuremath{\tensor*{z}{*^{(#2)}_{#1}}}} 
 \newcommand{\g}[2]{\ensuremath{\tensor*{g}{*^{(#2)}_{#1}}}} 
 \newcommand{\N}[2]{\ensuremath{\tensor*{N}{*^{(#2)}_{#1}}}} 
 \newcommand{\f}[2]{\ensuremath{\tensor*{f}{*^{(#2)}_{#1}}}} 
 \newcommand{\m}[2]{\ensuremath{\tensor*{\mathsf{m}}{*^{(#2)}_{[#1]}}}} 
 \newcommand{\U}[2]{\ensuremath{\tensor*{U}{*^{(#2)}_{#1}}}} 
  
 \newcommand{\kln}[1]{\ensuremath{\leq_{#1}}}
 \newcommand{\grn}[1]{\ensuremath{\geq_{#1}}}

 \newcommand{\vgrn}[1]{\rotatebox[origin=c]{180}{\ensuremath{\lessdot}}}

\begin{document}
\parindent0pt
\title{Monoid algebras of projection functors}
\author{Anna-Louise Grensing\\ Fakult\"at f\"ur Mathematik\\ Universit\"at Bielefeld\\ D - 36615 Bielefeld\\ agrensin@math.uni-bielefeld.de}
\maketitle
\begin{abstract} We study the monoid of so called projection functors $\p{S}$ attached to simple modules $S$ of a finite dimensional algebra, which appear naturally in the study of torsion pairs. We determine defining relations in special cases of path algebras. For the linearly oriented Dynkin quiver of Type $A$, we get an isomorphism to the monoid of non-decreasing parking functions. Moreover we give an explicit isomorphism between the monoid algebra of non-decreasing parking functions and a certain incidence algebra independent of the field.
\end{abstract}
\section{Introduction}
The study of $\mathrm{Hom}$-orthogonal subcategories is a classical tool in the representation theory of finite dimensional algebras. The process of passing from the module category to orthogonal subcategories is essential as pointed out for example in \cite{GeigleLenzing}, \cite{Schofield}. 
Here the so called projection functors appear naturally. We start by defining them for arbitrary modules $U$ over  an associative finite dimensional unital algebra $A$ over a field $k$. We denote by $\mo{A}$ the category of finite dimensional left $A$-modules and by $\gen{U}$ its full subcategory which consists of those modules isomorphic to a quotient of some $U^{\oplus d}$. We first define the endofunctor $\subf{U}$ on $\mo{A}$ by sending a module $M$ to its greatest submodule $\subf{U}(M)$ which lies in $\gen{U}$. This gives a subfunctor of the identity functor $\mathrm{id}_{\mo{A}}$, i.e.\ the embeddings $\iota_M$ of $\subf{U}(M)$ into $M$ for every module $M$ yield a natural transformation $\iota\colon \subf{U} \longrightarrow \mathrm{id}_{\mo{A}}$.
We obtain the so called projection functor $\p{U}\colon\mo{A}\longrightarrow \mo{A}$ by passing to the cokernel $(\pi_M)_{M\in \mo{A}}\colon \mathrm{id}_{\mo{A}}\longrightarrow \p{U}$ of $\iota$. 

The question of describing the relations between such functors arises naturally. It fits into the general categorification programmme of realizing Lie-theoretic objects as functors on module categories \cite{Mazorchuk}.  

We concretrize this and concentrate just on the multiplicative interplay between certain projection functors. To create the framework we consider the monoid $\monoid{A}$ generated by \mbox{$\set{\p{S}}{S \text{simple}}$} up to natural isomorphism. For this we recall that the composition of endofunctors on $\mo{A}$ induces a multiplication on the isomorphism classes of endofunctors. Hence $\monoid{A}$ is the set of the isomorphism classes $[\p{S_1}\p{S_2}\ldots\p{S_r}]$ with $r\geq 0$ and $S_1,S_2,\ldots, S_r$ simple $A$-modules together with that multiplication. In this sense we speak of monomials over $\set{\p{S}}{S \text{simple }}$, omit the brackets indicating the isomorphism classes and write $1$ for the isomorphism class of the identity functor.
To study this monoid and its monoid algebra a good way to start is finding a set of defining relations. We give some relations:
\begin{prop}\label{relations proposition}
 Let $S$ and $T$ be simple $A$-modules without any non-trivial self-extensions. Then the following relations hold:
 \begin{itemize}
  \item[(a)] $\p{S}\circ\p{S} = \p{S}$
  \item[(b)] If $\Ext{A}{1}{T,S}=0$ holds, we have $\p{S}\circ\p{T}\circ \p{S} = \p{T}\circ\p{S}\circ\p{T}= \p{S}\circ\p{T}$.
  \end{itemize}
\end{prop}
As a consequence of the second relations we get $\p{S}\circ\p{T}= \p{T}\circ\p{S}$ if $\Ext{A}{1}{T,S}=0=\Ext{A}{1}{S,T}$ holds. Thus certain generators of $\monoid{A}$ satisfy the braid relations of type $A$.

Now if $A$ is the path algebra $kQ$ of a finite acyclic quiver $Q=(Q_0,Q_1)$ over $k$, then each simple $kQ$-module $S_t$ (attached to the vertex $t$) has no non-trivial self-extensions and hence the proposition holds for all generators of $\monoid{Q}\ldef \monoid{kQ}$. We conjecture that these relations are defining for $\monoid{Q}$.
To determine whether the above relations are defining ones we compare $k\monoid{Q}$ with the algebra $B_Q$, which we define by generators $X_t$ with $t\in Q_0$ and the following relations: \begin{itemize}
 \item $X_s^2= X_s$ for all $s\in Q_0$
 \item $X_sX_tX_s=X_tX_sX_t$ for all $s,t \in Q_0$
 \item $X_tX_sX_t= X_sX_t$ for all $s,t \in Q_0$, such that there is no arrow from $t$ to $s$ 
\end{itemize}
There is an epimorphism $\psi_Q$ of algebras from $\B_Q$ onto $k\monoid{Q}$ with $X_t\mapsto \p{S_t}\rdef \p{t}$ for all $t\in Q_0$ by the above relations of the projection functors since the dimension of $\Ext{kQ}{1}{S_t,S_s}$ coincides with the number of arrows from $t$ to $s$. Note that if this epimorphism $\psi_Q$ is an isomorphism, then $\monoid{Q}$ is isomorphic to the Hecke-Kiselman semigroup associated with $Q$ introduced in \cite{MazorchukKiselman}. We introduce a method to detect when $\psi$ is an isomorphism in section 3. So far we applied it successfully to tree quivers with a specific orientation including bipartite tree quivers, $m$-subspace quivers, star quivers and each Dynkin quiver of type $A$ as well as to a couple of families of symmetrically shaped quivers. In this article we just discuss the former ones (see Prop.\ \ref{Special cases  prop normal form of B_n lin.oriented} and Theo.\ \ref{Special cases theo normal form of admissable trees}). As the relations are independent of the number of arrows unless there are none, one expects that multiple arrows have no impact. Moreover the relations are local as they just take direct neighbourhoods into account. We treat these aspects in section 3 for $\B_Q$ where it is obvious and for $k\monoid{Q}$.   

In section 4 we study the algebra $\B_Q$ and compute its Gabriel quiver. By definition its underlying monoid is the Hecke-Kiselman semigroup associated with $Q$. These algebras emerge as finite dimensional (see Cor.\ \ref{Special_cases Cor finite}) and basic (see Prop.\ \ref{extkoecher prop Radikal und alle Einfachen}) regardless of the representation type of $Q$. The Gabriel quiver can be described using the combinatorics given by the shape of the original quiver. 
\begin{theo}\label{extkoecher-theo-Extkoecher}
The simple modules $E_M$ of $\B_Q$, hence the vertices of the Gabriel quiver $\extko{\B_Q}$, are parametrized by the subsets $M$ of $Q_0$. Moreover there is at most one arrow between two vertices. More precisely we have $[\E_M]\rightarrow [\E_N] \in \extko{\B_Q}$ for two subsets $M,N$ of $Q_0$ if and only if $M\ohne N$ and $N\ohne M$ are non-empty and if for each pair $(m,n)\in M\ohne N\times N\ohne M$ there is an arrow $m\rightarrow n$ in $Q$. 
\end{theo}
Since the Hecke-Kiselman semigroup associated with an acyclic finite quiver is $\mathcal{J}$-trivial, these results can be deduced from \cite{DHST}. The author thanks Anne Schilling for pointing this reference out. Let $Q_n$ denote the following linearly oriented Dynkin quiver of type $A$:
 \begin{center}
 \begin{tikzpicture}[description/.style={fill=white,inner sep=2pt}, font=\normalsize]
            \matrix(m) [matrix of math nodes, row sep=1em, column sep=2em, text height=1.5ex, text depth=0.25ex]
            {1 & 2 &	{}\ldots & n\\};
	    \path[->,font=\scriptsize]
            (m-1-1) edge (m-1-2)
            (m-1-2) edge (m-1-3)
            (m-1-3) edge (m-1-4);
\end{tikzpicture}
\end{center}
This family of quivers has a special role, we show in subsection \ref{subsection properties of Gabriel quiver}:
\begin{prop} \label{Gabriel prop connected components}
Let $Q$ be connected, finite and acyclic.
\begin{itemize}
\item[(a)] If $Q=Q_n$ then the Gabriel quiver of $\B_{Q_n}$ has exactly $n+1$ connected components. 
\item[(b)] If $Q$ is distinct from the linearly oriented Dynkin quiver of type $A$, then the Gabriel quiver of $\B_Q$ has exactly $3$ connected components.
\end{itemize}
\end{prop}
Therefore we devote section 5 to the algebra $\B_{Q_n}$, which is isomorphic to $k\monoid{Q_n}$ by section 3. It can be read almost independent. Now the monoid algebra of non-decreasing parking functions $\mathrm{NDPF}_{n+1}$ has the same defining relations as $\B_{Q_n}$ (see \cite{Hivert} or \cite{MazorchukKiselman}) and is thus isomorphic to $\B_{Q_n}$. As shown in \cite{Hivert} it is isomorphic to the incidence algebra $\Inc{\Po_n}$ of the product order $\grn{n}$ on the powerset of $\ito{n}\ldef\{1,2,\ldots, n\}$ if the underlying field $k$ is the field of complex numbers $\CC$. We recall, that for two subsets $K=\{k_1<\ldots < k_r\}$ and $J=\{j_1<\ldots < j_m\}$ of $\ito{n}$ we have $K\grn{n} J$ iff $m=r$ and $k_i\geq j_i$ holds for all $i\in \ito{r}$.  They use the representation theory of the symmetric group. In \cite{DHST} this is generalized using the $\mathcal{J}$-triviality of the underlying monoid. Here (see Main Theorem \ref{Dynkin maintheorem}) we inductively construct an isomorphism from $\Inc{\Po_n}$ to $\B_{Q_n}$ independent of the field -- in fact it holds for a commutative ring -- using the structure of the tower of algebras $\B_{Q_1}\subset \B_{Q_2}\subset \ldots$ and the action of $\monoid{Q_n}$ on the (injective indecomposable) $kQ_n$-modules. 
We expect that a notion of non-decreasing parking functions for arbitrary quivers determines defining relations for $\monoid{Q}$.
\newline\textbf{Acknowledgments:}
\newline This work is part of my PhD-thesis \cite{Paasch} supervised by Markus Reineke whom I thank for posing this question. 
\section{Relations}
For every homomorphism $\varphi\colon M\longrightarrow N$ in $\mo{A}$ we get the following commutative diagram defining $\p{U}\varphi$, which we will call the diagram of $\varphi$ induced by $\p{U}$, with exact rows, which we will call the short exact sequence of $M$ resp. $N$ induced by $\p{U}$: 
\begin{equation*}
        \begin{tikzpicture}[description/.style={fill=white,inner sep=2pt}]
            \matrix(m) [matrix of math nodes, row sep=3em, column sep=4em, text height=1.5ex, text depth=0.25ex]
            { 	0 & \subf{U}M & M & \p{U}M = {M}/{\subf{U}M} & 0  \\
		0 & \subf{U}N & N & \p{U}N = {N}/{\subf{U}N} & 0\\ };
            \path[->,font=\normalsize]
            (m-1-1) edge  (m-1-2)
            (m-1-2) edge  node[above]{$\iota_M$} (m-1-3)
	    (m-1-2) edge node[auto, swap]{$\subf{U}(\varphi)=\varphi|_{\subf{U}M}$}  (m-2-2)
	    (m-1-3) edge node[auto, swap]{$\varphi$} (m-2-3)
            (m-1-3) edge   node[above]{$\pi_M$} (m-1-4) 
	    (m-1-4) edge  (m-1-5)
	    (m-1-4) edge [dotted] node[auto, swap]{$\p{U}(\varphi)$}  (m-2-4)
            (m-2-1) edge  (m-2-2)
	    (m-2-2) edge node[above]{$\iota_N$} (m-2-3)
	    (m-2-3) edge  node[above]{$\pi_N$} (m-2-4)
	    (m-2-4) edge  (m-2-5);
        \end{tikzpicture}
    \end{equation*} 
Other descriptions of the endofunctor $\subf{U}$ are useful. Obviously for every $A$-module $M$, the module $\subf{U}M$ is the sum over all submodules $X$ of $M$ lying in $\gen{U}$. Therefore $\subf{U}M$ is the image of the evaluation map \mbox{$\mathrm{ev}_{U,M}\colon U \otimes_{\End{A}{U}} \Hom{A}{U,M} \longrightarrow M$} with \mbox{$u\otimes \varphi \mapsto \varphi(u)$}.
Now if $U=S$ is simple then the module $\subf{S}M$ is isomorphic to some $S^{\oplus d}$ and \mbox{$\mathrm{ev}\colon S \otimes_{\End{A}{S}} \Hom{A}{S,\_} \longrightarrow \subf{S}$} is thus even a natural isomorphism. 
\begin{proof}[Proof of Proposition \ref{relations proposition}(a).]
If $S$ has no non-trivial self-extensions, then $\subf{S}$ is a torsion radical, i.e.\ a subfunctor of the identity functor such that \mbox{$\subf{S}\bigl(M/\subf{S}M\bigr) = 0$} holds for all $M\in \mo{A}$, and $\im\p{S}=\ker\Hom{A}{S,\_}$ holds. 
\end{proof}
We will prove (b) of Proposition \ref{relations proposition} with (and after) the following lemma. 
\begin{lem}
Let $S$ and $T$ be two non-isomorphic simple modules. Then $\p{T}\circ \p{S}$ and $\p{S\oplus T}$ are naturally isomorphic if and only if $\Ext{A}{1}{T,S}=0$ holds.
\end{lem}
\begin{proof}
 If there exists a non-split short exact sequence $0\longrightarrow S \overset{f}{\longrightarrow} X \longrightarrow T \longrightarrow 0$, then
the functors $\p{T}\circ \p{S} $ and $\p{S\oplus T}$ are not naturally isomorphic since we have:
\[
 \p{T}\circ \p{S} (X) = \p{T}\bigl(X/f(S)\bigr) = 0 \neq T \cong X/f(S) = \p{S\oplus T} (X)
\]
Now we assume $\Ext{A}{1}{T,S}=0$. Let $M$ be a module. Since $\subf{S\oplus T}M$ is semi-simple it coincides with $\subf{S}M \oplus \subf{T}M$. For that reason the restriction of the canonical projection \mbox{$\pi\colon M \longrightarrow M/\subf{S}M$} to $\subf{S\oplus T} M $ factors through \mbox{$\pi'\colon \subf{S\oplus T} M\longrightarrow \subf{T} (\p{S}M)$}. By passing to the cokernels we get the following commutative diagram whose exact rows are induced by $\p{S\oplus T}$ and $\p{T}$:
 \begin{equation*}
 \begin{tikzpicture}[description/.style={fill=white,inner sep=2pt}, font=\normalsize]
            \matrix(m) [matrix of math nodes, row sep=3em, column sep=2em, text height=1.5ex, text depth=0.25ex]
            {0 & \subf{S\oplus T} M			& M		& \p{S\oplus T}M ={M}/\subf{S\oplus T} M  	& 0\\
	    0  & \subf{T}(\p{S}M)		& \p{S}M 	& \p{T}\p{S}M = (\p{S}M)/\bigl(\subf{T}(M/\subf{S}M)\bigr)	& 0\\};
	    \path[->,font=\normalsize]
            (m-1-1) edge (m-1-2)
	    (m-1-4) edge (m-1-5)
	    (m-1-2) edge  (m-1-3)
            (m-1-2) edge node[auto] {$\pi'$} (m-2-2)
	    (m-1-3) edge  (m-1-4)  
            (m-1-3) edge node[auto] {$\pi$} (m-2-3)
	    (m-2-2) edge  (m-2-3)
            (m-2-3) edge  (m-2-4)
	    (m-2-1) edge (m-2-2)
	    (m-2-4) edge (m-2-5)
	    (m-1-4) edge [dotted] node[auto] {$\hat{\pi}=\hat{\pi}_M$} (m-2-4);
\end{tikzpicture}
\end{equation*}
To see that the epimorphism $\hat{\pi}_M$ is an isomorphism, it suffices to show that $\pi'$ is an epimorphism because of $\ker\pi'\cong \ker\pi$. Since $\Hom{A}{T,M}$ and $\Hom{A}{T,\p{S}M}$ are isomorphic by the assumptions, this follows from:
\[
 \subf{T}(\p{S}M) \cong T \otimes_{\End{}{T}} \Hom{A}{T,\p{S}M} \cong T \otimes_{\End{}{T}} \Hom{A}{T,M} \cong \subf{T}M\cong \pi'(\subf{S\oplus T}M) 
\]
Moreover we get for every morphism $\varphi\colon M\longrightarrow N$ the following cube,in which the sides on the left commute and in which thus the right square commutes. Thus \mbox{$\hat{\pi}$} is a natural isomorphism.
\begin{equation*}
\begin{tikzpicture}
  \matrix (m) [matrix of math nodes, row sep=1.5em, column sep=1.5em]
  { 	& N 	& 			&\p{S\oplus T}N &\\
 M  	& 	& \p{S\oplus T}M 	& 		&0\\
	& \p{S}N& 			& \p{T}\p{S}N 	&\\
 \p{S}M & 	& \p{T}\p{S}M 		& 		&\\};
  \path[-stealth, font=\scriptsize]
    (m-1-2) edge (m-1-4) 
            edge [densely dotted] (m-3-2)
    (m-1-4) edge node[above=15pt, right] {$\hat{\pi}_N$} (m-3-4) 
    (m-2-1) edge [-,line width=6pt,draw=white] (m-2-3)
    (m-2-1) edge  (m-2-3) edge (m-4-1) 
    (m-2-1) edge node[auto, swap] {$\varphi$} (m-1-2) 
    (m-3-2) edge [densely dotted] (m-3-4)
    (m-4-1) edge [densely dotted] node[auto, swap] {$\p{S}\varphi$}(m-3-2)
    (m-4-1) edge (m-4-3)
    (m-2-3) edge [-,line width=6pt,draw=white] (m-4-3)
            edge node[above=15pt, right] {$\hat{\pi}_M$}(m-4-3)
    (m-2-3) edge node[auto, swap] {$\p{S\oplus T}\varphi$} (m-1-4)
    (m-2-3) edge [-,line width=6pt,draw=white] (m-2-5)
	    edge (m-2-5)
    (m-4-3) edge node[auto, swap] {$\p{T}\p{S}\varphi$} (m-3-4);
\end{tikzpicture}
\end{equation*}
\end{proof}
\begin{proof}[Proof of Proposition \ref{relations proposition}(b).]
By the previous lemma it suffices to show \newline\mbox{\quad(A)\quad $\p{S\oplus T}\circ \p{T} \sim \p{S}\circ\p{T}$}\quad and \mbox{\quad(B)\quad $\p{S}\circ\p{S\oplus T} \sim \p{S}\circ\p{T}$}.
\newline\textit{Proof of (A).} Since $T$ has no non-trivial self-extensions and is as simple as $S$, a natural isomorphism is induced by:
\[
 \subf{S\oplus T}(\p{T}M)\cong \subf{S}(\p{T}M)\oplus \subf{T}(\p{T}M) \cong \subf{S}(\p{T}M)
\]
\newline\textit{Proof of (B).} Let $\alpha\colon \p{T}\longrightarrow \p{S\oplus T}$ be the natural transformation given by the following composition $\alpha_M$ of the canonical epimorphism and isomorphism for every module $M$:
\begin{equation*}
\begin{tikzpicture}
            \matrix(m) [matrix of math nodes, row sep=3em, column sep=3.5em, text height=1.5ex, text depth=0.25ex]
	    {\p{T}M={M}/{\subf{T}M} &(M/\subf{T}M)\bigg/\biggl((\subf{S}M\oplus \subf{T}M)/\subf{T}M\biggr) &  \p{S\oplus T} M \\};
	    \path[->>, font=\normalsize]
		(m-1-1) edge (m-1-2);
	    \path[->, font=\normalsize]	
		(m-1-1.north east) edge [bend left = 12] node[auto]{$\alpha_M$} (m-1-3.north west)
		(m-1-2) edge node[above]{$\cong$} (m-1-3);    
	    \end{tikzpicture} 
\end{equation*}
We claim, that the natural transformation \mbox{$(\p{S}\alpha_M)_{M\in \mo{A}}\colon \p{S}\circ\p{T}\longrightarrow \p{S}\circ \p{S\oplus T}$} is a natural isomorphism. For this we look at the diagram of $\alpha_M$ induced by $\p{S}$ and consider its exact sequence of kernels and cokernels given by the snake lemma:
\[
0\rightarrow \ker{\subf{S}\alpha_M} \overset{\iota}{\longrightarrow} \ker{\alpha_M}\longrightarrow \ker{\p{S}\alpha_M}
\longrightarrow \coker{\subf{S}\alpha_M} \longrightarrow 0 \longrightarrow \coker{\p{S}\alpha_M}\rightarrow 0
\]
Now $\ker \alpha_M\cong \subf{S}M\in \gen{S}$ is a submodule of $\subf{S}\p{T}M$. Hence the monomorphism $\iota$ is an isomorphism. Moreover $\subf{S} \alpha_M$ is surjective. This is seen by using the assumptions on $S$, which yield the surjectivity of $\Hom{A}{S,\alpha_M}$, and the natural isomorphism of $\subf{S}$ and $S\otimes \Hom{A}{S,\_}$.
\end{proof}
\section{The monoid algebras $\mathbf{k\monoid{Q}}$ and $\mathbf{\B_Q}$ for path algebras $\mathbf{kQ}$}
We now turn towards the class of finite dimensional path algebras. So let $Q=(Q_0,Q_1)$ be a finite acyclic quiver, i.e.\ $Q$ is an oriented graph without oriented cycles and with finite sets $Q_0$ and $Q_1$ of vertices and arrows respectively. We will denote an arrow in $Q$ by $\alpha\colon s\rightarrow t \in Q_1$ or $s\overset{\alpha}{\rightarrow} t$,  and a path $t_n\overset{\beta_n}{\leftarrow}\ldots\overset{\beta_{2}}{\leftarrow}t_{1}\overset{\beta_1}{\leftarrow}t$ by $\beta_n\ldots\beta_{2}\beta_1$. The path algebra of $Q$ over a field $k$ is denoted by $kQ$. The category $\mo{kQ}$ and the category $\rep{k}{Q}$ of finite dimensional $Q$-representations over $k$ are equivalent and we will not distinguish between them. 
Let $(\free{Q},\konk)$ be the free monoid over $Q_0$, i.e.\ the monoid of words over the alphabet $Q_0$ with the concatenation $\konk$ as multiplication. There is the canonical epimorphism $\rho\colon k\free{Q} \longrightarrow \B_Q$ with $q\mapsto X_q$ for all $q\in Q_0$. We will denote the image of a word $w\in \free{Q}$ under $\rho$ by $X_w$. By $\p{w}$ we denote the image of $X_w$ under the canonical epimorphism $\psi_Q\colon \B_Q\longrightarrow k\monoid{Q}$. 
\begin{df}
We call a subset $W$ of $\free{Q}$ an admissible normal form associated with $Q$ if the following conditions hold:
\begin{itemize}
 \item[(1)] $\{\emptyset\} \cup Q_0 \subseteq W$
 \item[(2)] $B_W\ldef \{X_w\mid w\in W\}$ is closed under (right-) multiplication with the generators $X_t$ of $\B_Q$.
 \item[(3)] For all words $v\neq w$ in $W$, there is a $Q$-representation $V$ with $\p{v} V \cong\!\!\!\! \left\vert\right.\; \p{w} V$.
\end{itemize}
\end{df}
Obviously the set $\set{X_w}{w\in \free{Q}}$ is a ${k}$-linear generating system of $\B_Q$. This definition extracts suitable conditions on a subset of $\set{X_v}{v\in \free{Q}}$ to be a basis of $\B_Q$ forcing $\psi_Q$ to be an isomorphism: due to the conditions (1) and (2), $B_W$ is a submonoid of $\B_Q$ which contains the generators $X_q$ of $\B_Q$ and the unit $1$. Hence the $k$-linear span of $B_W$ is the monoid algebra $\B_Q$ itself. Furthermore condition (3) ensures that the elements of $B_W$ are indexed by $W$. Thus $B_W$ is a $k$-linear basis of $\B_Q$ with $|W|$ elements. Moreover the canonical epimorphism $\psi_Q$ is an isomorphism because of condition (3).
The existence of an admissible normal form is not obvious. But finding one summarizes our strategy of proving that the relations are defining in several special cases.
\subsection{Tools for (2) and reductions}\label{subsection tools for (2)}
To begin with one needs a better understanding of the multiplication of two arbitrary monomials in $\B_Q$, that is to say of the defining relations. Since $Q$ is acyclic the third relation (under the two first ones) is equivalent to the following two:
\begin{itemize}
\item $X_tX_sX_t= X_sX_t$ for all $s,t \in Q_0$, such that there is an arrow $\alpha\colon s\rightarrow t$
\item $X_sX_t= X_tX_s$ for all vertices $s,t \in Q_0$ which are not connected by an arrow
\end{itemize}
Therefore the underlying monoid of $\B_Q$ is isomorphic to the Hecke-Kiselman semigroup associated with $Q$ introduced in \cite{MazorchukKiselman}. Let us fix some more notation. We will write $\{v\}$ for the set of the letters occuring in the word $v\in \free{Q}$. For example we have $\{v\}= \{1,5,7,15\}$ if $Q_0=\{1,2,\ldots, 15\}$ and $v= 1\konk5\konk15\konk7\konk1\konk5$. For a subset $M$ of $Q_0$ we will denote by $Q_M$ the full subquiver of $Q$ whose underlying set of vertices is exactly $M$ and we will abbreviate $Q_v\ldef Q_{\{v\}}$.  A vertex $t \in Q_0$ is called a sink (source), if no arrow has tail (head) $t$. The condition on $s,t\in Q_0$ in the third relation defining $\B_Q$ could be replaced by requiring $t$ to be a sink in $Q_{s\konk t}$. The defining relations of $\B_Q$ generalize to the following identities in $\B_Q$ (and thus in $k\monoid{Q}$):
\begin{lem}\label{Normalform-lem-generalized relations}
For all $t\in Q_0$ and all words $v,w$ over $Q_0$ we have:
\begin{itemize}
 \item[] $X_tX_wX_t=X_wX_t$ \qquad\, if $t$ is a sink in the subquiver $Q_{t\konk w}$
 \item[] $X_tX_wX_t=X_tX_w$ \qquad\: if $t$ is a source in the subquiver $Q_{t\konk w}$ 
\item[] $X_vX_w=X_wX_v$ \qquad\: if there is no arrow between the subquivers $Q_{v}$ and $Q_{w}$
\end{itemize}
\end{lem}
\begin{proof}
The second identity follows by duality from the first one, meanwhile the third one results directly from the others. We prove the first identity by an induction on the length of $w$.  If $w=s$ the identity is just one of the defining relations of $\B_Q$. So let $w=u\konk s$ for some $s\in Q_0$ and some word $u$. The induction hypothesis applies now to $s$ and to $u$:
 \[
  X_tX_wX_t= X_tX_uX_sX_t = X_tX_uX_tX_sX_t= X_uX_tX_sX_t= X_uX_sX_t=X_wX_t
 \]
\end{proof}
Essentially all calculations are abbreviated using these generalized relations. For example the finiteness of $\monoid{Q}$ can be deduced from them. 
\begin{cor}\label{Special_cases Cor finite}
The $k$-algebra $\B_Q$ is finite dimensional. Hence the monoid $\monoid{Q}$ is finite.
\end{cor}
\begin{proof}
If $Q$ just consists of one vertex, $\B_Q$ is two-dimensional. So now we assume $Q$ to have at least two vertices, pick a sink $s\in Q_0$ and consider the quiver $K\ldef Q_{Q_0\setminus \{s\}}$. Inductively $\B_K\subseteq \B_Q$ is finite dimensional, hence has a finite basis $B$ of monomials over $\set{X_t}{t\in Q_0\!\setminus \!\{s\}}$. Due to Lemma \ref{Normalform-lem-generalized relations} every monomial $X_w \in \B_Q$ lies either in $\B_KX_s\B_K$ (if $s\in \{w\}$) or in $\B_K$ (if $s\notin \{w\}$). Thus $B\cup BX_sB$ is a finite $k$-linear generating system of $\B_Q$.
\end{proof}
The next observations enable us to restrict to isomorphism classes of finite, acyclic and connected quivers without multiple arrows. We will call a subquiver $K=(K_0,K_1)$ of $Q$ ``the quiver reduced by multiple arrows of $Q$`` if $K_0=Q_0$ holds and if there is exactly  one arrow between two vertices $s$ and $t$ in $K_1$ whenever there exists (at least) one arrow between $s$ and $t$ in $Q_1$. 
\begin{prop}\label{specialcases prop BK and BQ} 
Let $K, K'$ and $Q$ be finite, acyclic quivers.
\newline (a)\quad If $K$ is a full subquiver of $Q$, then $\B_K$ is a subalgebra of $\B_Q$. 
\newline (b)\quad If $Q$ and $K$ are (anti-) isomorphic, then $\B_Q$ and $\B_K$ are (anti-) isomorphic. 
\newline (c)\quad If $K$ is the quiver reduced by multiple arrows of $Q$, then $\B_K$ and $\B_Q$ are isomorphic. 
\newline (d)\quad If $K$ and $K'$ are the connected components of $Q$, then $\B_{Q}$ and $\B_{K}\otimes_k \B_{K'}$ are isomorphic.
\end{prop}
The analogous statements for $k\monoid{K}$ and $k\monoid{Q}$ are a priori not clear. However, if there is an admissible normal form associated with $K$ most of them hold (see Proposition \ref{specialcases prop monoidK and monoidQ}). 

\subsection{Tools for (3) and a criteria for reductions}\label{subsection tools for (3)}
For any vertex $t\in Q_0$ we abbreviate $\p{t}= \p{t}^{(Q)}\ldef \p{S_t}$ and similar for $\subf{S_t}$.
The projection functor $\p{t}$ on $\rep{k}{Q}$ associated with $S_t$ is  easily computed for representations $V$ of $Q$, since the functor $\subf{S_t}$ maps $V$ to that submodule $U$ of the socle $\mathrm{Soc}(V)\subseteq V$, which is given by $U_t= (\mathrm{Soc}V)_t$  and $U_s=0$ for all $s\in Q_0\ohne\{t\}$. Therefore $\p{t}V$ is described by $(\p{t}V)_t=V_t/U_t$ and $(\p{t}V)_s=V_s$ for all $s\in Q_0\ohne\{t\}$  and the respectively induced $k$-linear maps. Consequences, in particular for simples and injective indecomposable representations, are summarized in the next remarks.  Let $I_x$ be the injective envelope of $S_x$.
\begin{rem}
A $Q$-representation $V$ is fixed under the action of those $\p{t}$ with $t\notin \mathrm{supp} V \linebreak\ldef \set{q\in Q_0}{V(q)\neq 0}$. Thus $\p{w}{V}= V$ holds for all words $w$ over $Q_0\ohne\mathrm{supp}(V)$. In particular we have for all $w\in \free{Q}$:
\[
\p{w}(S_t) = \begin{cases}
             S_t \quad &\text{if }\; t\notin \{w\}\\ 
             0 \quad &\text{if }\; t\in \{w\}\\
\end{cases}
\]
\end{rem}
\begin{rem}
The action of $\p{t}$ on the injective indecomposable representation $I_x$ is:
\[
 \p{t}(I_x) = \begin{cases}
               I_x\quad &\text{if }\; x\neq t\\
               I_t/S_t = \bigoplus_{s\rightarrow t} I_s \quad&\text{if }\; x=t
              \end{cases}
\]
Hence $I_x$ is fixed under $\p{w}$ for all words $w$ over $Q_0$ not containing $x$. 
\end{rem}
Let $K$ be a subquiver of $Q$ and $F\colon\rep{k}{K} \rightarrow \rep{k}{Q}$ the canonical embedding functor. Recall that for every $K$-representation $U$ the $Q$-representation $FU$ is defined by setting for all $t\in Q_0$ and $\alpha\in Q_1$:
\[
(FU)_t \ldef \begin{cases}
              U_t \quad &\text{if } t\in K_0\\
              0\quad &\text{otherwise}
             \end{cases} \qquad \text{and } \quad (FU)_\alpha \ldef \begin{cases}
              U_\alpha \quad &\text{if } \alpha\in K_1\\
              0\quad &\text{otherwise}
             \end{cases}
\]
\begin{lem}\label{Special_cases lem embedding functor}
Let $K$ be a subquiver of $Q$, $x\in K_0$ and $v, w\in\free{K}$.
\newline (a)\quad The functors $F\p{x}^{(K)}$ und $\p{x}^{(Q)}F$ are naturally isomorphic.
\newline (b)\quad If there is a $K$-representation $V$ with \mbox{$\p{v}^{(K)}V\ncong \p{w}^{(K)}V$}, then $\p{v}^{(Q)}$ and $\p{w}^{(Q)}$ are not naturally isomorphic.
\end{lem}
\begin{proof}
We apply $F$ to the short exact sequence of a $K$-representation $U$ induced by $\p{x}^{(K)}$ and get the short exact sequence:
\[
0\longrightarrow F\subf{x}^{(K)} U \longrightarrow FU \longrightarrow F\p{x}^{(K)} U\longrightarrow 0
\]
This is already the short exact sequence  of $FU$ induced by $\p{x}^{(Q)}$ because of \mbox{$F\subf{x}^{(K)} U=\subf{x}^{(Q)}FU$} and the uniqueness (up to isomorphism) of the cokernel.
Therefore (a) holds. In particular, we conclude (b) from: $\p{v}^{(Q)}FV \cong F\p{v}^{(K)} V \ncong F\p{w}^{(K)} V \cong \p{w}^{(Q)}F V$. 
\end{proof}
\begin{prop} \label{specialcases prop monoidK and monoidQ}
Let $K,K'$ and $Q$ be finite, acyclic quivers. Assume that there are admissible normal forms $W$ and $W'$ associated with $K$ and $K'$ respectively. 
\newline (a)\quad If $K$ is a full subquiver of $Q$, then $k\monoid{K}$ is a subalgebra of $k\monoid{Q}$. 
\newline (b)\quad If $Q$ and $K$ are isomorphic, then $k\monoid{K}$ and $k\monoid{Q}$ are isomorphic. 
\newline (c)\quad If $K$ is the quiver reduced by multiple arrows of $Q$, then $k\monoid{K}$ and $k\monoid{Q}$ are isomorphic. 
\newline (d)\quad If $K$ and $K'$ are the connected components of $Q$, then $k\monoid{Q}$ and $k\monoid{K}\otimes_k k\monoid{K'}$ are isomorphic.
\end{prop}
\begin{proof}
By the assumption on $K$ we have $k\monoid{K}\cong \B_K$. In case (a) or (c), the assignment \linebreak\mbox{$\p{x}^{(K)}\mapsto \p{x}^{(Q)}$} for all $x\in K_0$ extends to a homomorphism $k\monoid{K}\longrightarrow k\monoid{Q}$. Its image is \linebreak linearly spanned by the monomials $\p{v}^{(Q)}$ with $v\in W$. But by the previous lemma, Lemma \ref{Special_cases lem embedding functor}, the \linebreak tuple $(\p{v}^{(Q)}\:|\:v\in W)$ is also linearly independent. Thus the assertions (a) - (c) hold for dimensional reasons. Meanwhile an admissible normal form associated with $Q$ in case (d) is given by $W\konk W'$. Therefore $k\monoid{Q}\cong \B_{Q}$ holds and (d) follows with Proposition \ref{specialcases prop BK and BQ}(d).
\end{proof}
\subsection{The linearly oriented Dynkin quiver of type $\mathbf{A}$}
We denote by $\Po_n$ the poset of the product order $\grn{n}$ on the powerset of $\ito{n}$ defined in the introduction. For two intervals $J$ and $I$ in $\ito{n}$ we define $J\groint I$ by requiring $\min J > \min I$  and  $\max J > \max I$, in particular $\{\max J, \max I\}\grn{n} \{\min J, \min I\}$. \label{specialcases_groint} The poset $\Po_n$ is in bijection with the set consisting of tuples of intervals $J_r \groint \ldots \groint J_1$ in $\ito{n}$.  For an interval $J= \{i,i+1,\ldots,j-1,j\}$ of positive integers let $J$ denote the word $i\konk i+1\konk \ldots\konk j-1\konk j$ as well. The monomial $X_J\in \B_{Q_n}$ is an idempotent. More precisely for all $k\in J$ we have $X_JX_k = X_J$ since $k$ is a source in the subquiver $Q_{k\konk k+1\konk\ldots \konk j-1\konk j}$ (Lemma \ref{Normalform-lem-generalized relations}). We will meet a generalisation of these idempotents to arbitrary finite quivers without oriented cycles to determine the radical of $\B_Q$ in the next section. 

\begin{prop} \label{Special cases  prop normal form of B_n lin.oriented}
An admissible normal form of $\B_{Q_n}$ is
\[
W_n \ldef \set{J_r\konk \ldots \konk J_1}{r\in \oto{n},\:\: J_r \groint \ldots \groint J_1 \:\text{ intervals in }\: \ito{n}}  
\]
\end{prop}
So $\B_{Q_n}$ and $k\monoid{Q_n}$ are isomorphic and have the dimension $C_{n+1}=\frac{1}{n+2}\binom{2(n+1)}{n+1}$ by characterization 6.19.aa. in \cite{Stanley} of the  $n+1^{\mathrm{th}}$ Catalan number $C_{n+1}$. 
\begin{proof}
Condition (1) is easily verified by looking at $r=0$ and $r=1$. To check condition (2) we just state a multiplication rule: 
\bigskip
\newline\textit{Let $n\in \NN$. For all intervals $J$ and $L_1\kleint \ldots \kleint L_s$ in $\ito{n}$ and for $L_0\ldef \emptyset \rdef L_{s+1}$  we define indices $y=y(J,L_1,\ldots,L_s)$ and $ z=z(J,L_1,\ldots,L_s)$ by:
\[
\begin{aligned}
 y &\ldef \begin{cases}
         \max \set{x\in \ito{s}}{J\groint L_x \text{ and } J\cap(L_x +1) = \emptyset}\quad&\text{ if } L_1\kleint J \text{ and } J\cap(L_1 +1) = \emptyset\\
	 0\quad&\text{ else}
	 \end{cases}\\
\text{and} &\\
 z &\ldef \begin{cases}
	y+1\quad&\text{ if } J\kleint L_{y+1} \\                  
        \min \set{x\in \{y+2,\ldots,s\}}{ L_{y+1} \cup J \kleint L_x}\quad&\text{ if } J \nkleint L_{y+1} \text{ and } J\cup L_{y+1} \kleint L_s\\
	s+1\quad&\text{ else, thus if } J \nkleint L_{y+1} \text{ and } J\cup(L_{y+1}) \nkleint L_s\\
	    \end{cases}\\
\end{aligned}
\]
Then the product of $X_{ {L_1}\ldots  {L_s}}$ and $X_{J}$ is given by:
\[
\begin{aligned}
X_{ {L_s}\ldots  {L_1}}X_{J}&= X_{ {L_s}\ldots  {L_z} {(J\cup \bigcup_{t=y+1}^{z-1} L_t)}  {L_y}\ldots  {L_1}}
&=\begin{cases}
  X_{ {L_s}\ldots  L_{y+1} {J}  L_{y}\ldots  {L_1}}\quad&\text{if } J\kleint L_{y+1}\\  
  X_{ {L_s}\ldots  {L_z} {(J\cup L_{y+1})}  {L_y}\ldots  {L_1}}\quad&\text{otherwise}
  \end{cases}\quad 
\end{aligned}
\]
}
The proof is a lengthy but straightforward induction on $s$ requiring case by case analysis. 
Now we turn towards condition (3). For this it suffices to consider the injective indecomposable $Q_n$-representations $I_j$ for $j\in \oto{n}$. We set $I_0\ldef 0$. 
Because of $\p{j}I_j=I_{j-1}$ we get inductively on $r$ for all intervals $J_r\groint\ldots\groint J_1$ in $\ito{n}$:
\[
 \p{J_r\ldots J_1}I_j = \begin{cases}
                         I_{\min{J_a} -1} \quad &\text{if } j\in \bigcup_{x=1}^r J_x \text{, where } a\in \ito{r} \text{ is min. with } j\in J_a\\
			I_j &\text{otherwise}	  
\end{cases}
\]
Let $J_r\groint \ldots \groint J_1$ and $L_s\groint\ldots \groint L_1$ be two distinct tuples of intervals in $\ito{n}$. Without loss of generality we can pick an index $a\in\ito{r}$ minimal with respect to $J_a$ being distinct from $L_1,\ldots, L_s$. Then we have:
\begin{align*}
 \p{J_r\ldots J_1}( I_{\max J_a}) &= I_{\min{J_a} -1}\\
\shortintertext{and}
\p{L_s\ldots L_1}( I_{\max J_a}) &= \begin{cases}
                         I_{\min{L_b} -1} \: &\text{if } \max J_a \in \bigcup_{x=1}^s L_x \text{, where } b\in \ito{s} \text{ is min. with } \smash{\max J_a \in L_b}\\
			I_{\max J_a} &\text{otherwise}	  
\end{cases}
\end{align*}
We are done if $\max J_a \notin \bigcup_{x=1}^s L_x$. So we assume $\max J_a \in \bigcup_{x=1}^s L_x$. We are done as well if $\min{L_b} \neq \min{J_a}$. Thus let $\min{L_b} = \min{J_a}$. Now we consider the action on $I_{\max L_b}$:
\begin{align*}
\p{L_s\ldots L_1}I_{\max L_b} &= I_{\min L_b -1} = I_{\min J_a -1} \\
\shortintertext{and}
\p{J_1\ldots J_r} I_{\max L_b} &= \begin{cases}
        I_{\min J_c-1} \quad &\text{if } \max L_b \in \bigcup_{x=1}^r J_x \text{, where } c \in \ito{r} \text{ is min. with } \smash{\max L_b \in J_c}\\
	I_{\max L_b} \quad & \text{otherwise}
       \end{cases}\\
\end{align*}
In the first case (for $\max L_b$ ) the inequality holds because of  $\max J_c \geq \max L_b > \max J_a$, so $J_c\groint J_a$, hence $\min J_c > \min J_a$. In the second one it holds due to $\max{L_b} > \min L_b-1$.
\end{proof}
The defining relations for $\monoid{Q_n}$ are the same as for the monoid $\mathrm{NDPF}_{n+1}$ of non-decreasing parking functions (see \cite{Hivert} or \cite{MazorchukKiselman}) which is generated by the functions $\pi_j\ldef\binom{1\ldots j-1\, j \, j+1\ldots n+1}{1\ldots j-1\, j \, j\hphantom{+1}\ldots n+1}$. In fact, if $\beta$ is the bijection $\{I_0,I_1,\ldots,I_n\}\longrightarrow \ito{n+1}, I_j\mapsto j+1$ the proof of condition (3) yields the isomorphism $\monoid{Q_n}\longrightarrow \mathrm{NDPF}_{n+1}, \p{v}\mapsto \beta\p{v}\beta^{-1}$, which is the extension of the assignment $\p{j}\mapsto \pi_j$.  Thus the full subcategory of the category of covariant functors on $\mo{kQ}$ containing the elements $\p{t_1}\p{t_2}\ldots\p{t_r}$ for $r\geq 0$ and vertices $t_1,t_2,\ldots, t_r\in Q_0$ categorifies the monoid $\mathrm{NDPF}_{n+1}$.  
\subsection{Gluing on a sink}
We start with $n$ finite acyclic and pairwise disjoint quivers $Q(1),\ldots, Q(n)$ and pick some vertices \linebreak$p_1^{(1)},\ldots,p_{r_1}^{(1)} \in Q(1),\ldots, p_1^{(n)},\ldots,p_{r_n}^{(n)} \in Q(n)$. Then we consider the quiver $Q$ arising from gluing these quivers together on a new vertex $s$ over new arrows $\alpha_i^{(j)}\colon p_i^{(j)}\rightarrow s$. The shape of $Q$ is sketched below. We will moreover denote by $\Q(j)$ the full subquiver of $Q$ with the vertices $Q(j)_0\cup \{s\}$. 
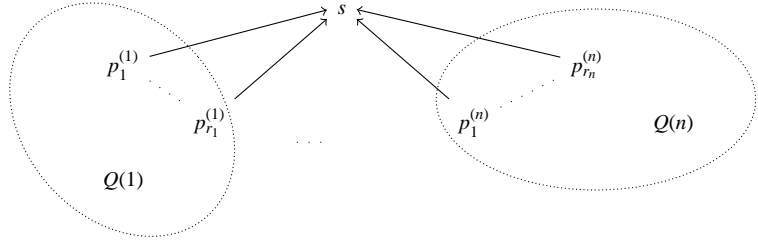
\begin{equation*}
 \begin{tikzpicture}
            \matrix(m) [matrix of math nodes, font=\scriptsize, row sep=0.2cm, column sep=0.4cm, text height=1.5ex, text depth=0.25ex]
            { 		&		&		&		&s				&  		&		&		&		&		&\\
			&p_1^{(1)}	&		&		&				&  		&		&		&		&p_{r_n}^{(n)}	&\\
			& 		& p_{r_1}^{(1)} &\hphantom{h}	&\hphantom{h}			&		&		& p_1^{(n)}	&		&		&Q(n) \\
			&Q(1)		&		&		&				& 		&		&		&		&		&\\
			&		&		&		&				&	 	&		&		&		&		&\\};
          \path[->]
 	    (m-2-2) edge (m-1-5) 
 	    (m-3-3) edge (m-1-5)	
	    (m-2-10) edge (m-1-5)	
 	    (m-3-8) edge (m-1-5);
	    \path
             (m-3-8) edge[loosely dotted] (m-2-10)
             (m-3-4.south east) edge[loosely dotted] (m-3-5.south west)
	    (m-2-2) edge[loosely dotted] (m-3-3);
\def\firstellipse{(0,0) ellipse (1.3cm and 1.7cm)}
\def\secondellipse{(0,0) ellipse (1.2cm and 2.1cm)}
\begin{scope}[rotate=40]
\draw[densely dotted,xshift=-2.8cm,yshift=2cm]\firstellipse;
\end{scope};
\begin{scope}[rotate=90]
\draw[densely dotted,xshift=0cm,yshift=-2.8cm]\secondellipse;
\end{scope};
   \end{tikzpicture}
\end{equation*}
\begin{lem}\label{Normalformen das Lemma Induktionsargument}
Assume that for every $j\in \ito{n}$ we have admissible normal forms $\W(j)\subseteq \free{Q(j)}$ and $\tW(j)\subseteq \free{\Q(j)}$ associated with $Q(j)$ and $\Q(j)$ respectively, such that  $\W(j)\subseteq \tW(j)$ holds and $s$ appears at most once in any word of $\tW(j)$.
Then an admissible normal form $W$ associated with $Q$ consists of the $\prod_{j\in\ito{n}} |\W(j)| + \prod_{j\in\ito{n}} (|\tW(j)|-|\W(j)|)$ words:
\begin{align*}
\text{(a)}\qquad w_1\konk w_2\konk\ldots\konk w_n \qquad &\text{with }\: w_j\in \W(j) \:\text{ for all }\: j\in \ito{n}\\
\shortintertext{and}
\text{(b)}\qquad y_1\konk y_2\konk \ldots\konk y_n \konk s \konk z_1\konk z_2\konk \ldots\konk z_n \qquad & \text{with }\: y_j\konk s \konk z_j \in \tW(j)\setminus\! \W(j)\: \text{ for all }\: j\in \ito{n}
\end{align*}
Thus $\B_Q \cong k\monoid{Q}$ holds. Moreover the dual assertions holds.
\end{lem}

\begin{proof}
Conditions (1) and (2) are verified straightforwardly using the generalized relations. For condition (3) it suffices, by the remarks in subsection \ref{subsection tools for (3)}, to  consider two words $v\neq w \in W$ both being either of type (a) or of type (b). In each case there are an index $j\in \ito{n}$ and subwords $v_j, w_j\in \tW(j)$ of $v$ and $w$ respectively such that $v_j\neq w_j$ holds. Now employ the assumption on $\tW(j)$ and Lemma \ref{Special_cases lem embedding functor}.
\end{proof}
A direct application of this lemma gives an admissible normal form associated with the $m$-subspace quiver $T_m$, which is the connected quiver with exactly one sink $s$ and $m$ sources enumerated by $1,\ldots, m$: here $Q(j)$ corresponds just to the vertex $j$ and $\Q(j)$ to $j\rightarrow s$. An admissible normal form associated to the latter is $\{\emptyset, j, s, js, sj\}$ which contains the normal form associated with $j$, i.e.\ $\{\emptyset, j\}$. To fix an order on the vertices we denote for every subset $J= j_1<\ldots <j_k$ of $\ito{m}$ the word $j_1\konk\ldots \konk j_k$ with $w(J)$. Now an admissible normal form associated with $T_m$ has $2^m+3^m$ elements:
\[
 \set{w(J)}{J\subseteq\ito{m}}\: \cup\: \set{w(I)\konk s\konk w(J) }{I,J \subseteq \ito{m}, I\cap J = \emptyset}
\]
This can be extended to the star quiver, since we have admissible normal forms associated with its branches, i.e linearly oriented Dynkin quivers of type $A$. 

\subsubsection{Tree quivers with a specific orientation}
Furthermore, Lemma \ref{Normalformen das Lemma Induktionsargument} is the induction step for tree quivers with a specific orientation: we will call $Q$ an admissible tree quiver, if each crossing of $Q$, i.e.\ a vertex whose entry degree or exit degree is at least $2$, is either a sink or a source. (So the linearly oriented Dynkin quivers $D_4$ are not subquivers of $Q$.) One can endow every tree with an orientation to obtain an admissible tree quiver. In particular every tree with a bipartite orientation is an admissible tree quiver. Note also that every Dynkin quiver of type $A$ is an admissible tree quiver. 

\begin{theo}\label{Special cases theo normal form of admissable trees}
There is an admissible normal form associated with any admissible tree quiver $Q$. In particular the relations are defining for $\monoid{Q}$. 
\end{theo}
\begin{proof}\label{Normalformen zulaessige Baumkoecher}
The case $Q$ being $Q_n$ for some $n$ is already done. So assume $Q$ not to be a linear oriented Dynkin quiver of type $A$. In particular we can pick a crossing $s$ of $Q$. To apply Lemma \ref{Normalformen das Lemma Induktionsargument} we just have to check, whether the assumptions hold for the subquivers which are linked to $s$ by one arrow. These subquivers (and their extensions with $s$) are again admissible tree quivers. So it suffices to show: 
\newline\textit{Let $K$ be an admissible tree quiver, $y\in K_0$ and $\K$ an extension of $K$ by one (new) vertex $s$ and one (new) arrow $ y\leftarrow s$ or $y\rightarrow s$, so that $\K$ is again an admissible tree quiver. 
Then there are admissible normal forms $\W\subseteq \free{K}$ and $\tW\subseteq \free{\K}$ associated with $K$ and $\K$ respectively, such that $\W\subseteq \tW$ and $s$ appears at most once in any word of $\tW$.}
\newline This is proven by induction on the number of vertices of $K$ using Lemma \ref{Normalformen das Lemma Induktionsargument}. 
\end{proof}
The induction step, i.e.\ Lemma \ref{Normalformen das Lemma Induktionsargument}, provides a procedure for gaining an admissible normal form associated with $Q$. We illustrate this by the special case of bipartite Dynkin quivers $K_n$ of type $A_n$. Depending on whether $n$ is even or odd $K_n$ has up to anti-isomorphism one of the following shapes: 
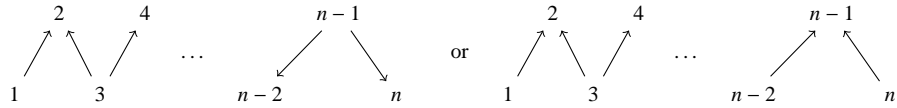
\begin{equation*}
 \begin{tikzpicture}
            \matrix(m) [matrix of math nodes, row sep=0.3em, column sep=0.5em, font=\scriptsize]
             {& 2 &   & 4 &           &     		& \vphantom{n}n-1 	&  & &		&   & 2 &   & 4 &           &     & n-1 &\\
	      &   &   &   &\ldots    &      		&  			&  & &	\text{or}	&   &   &   &   &\ldots    &     &  	&	\\
             1&   & 3 &   &         & \vphantom{n-2}n-2 	& 			&n & &		&1  &   & 3 &   &         & n-2 	& 			&n\\};
            \path[->, font=\scriptsize]
	    (m-3-1) edge (m-1-2)
	    (m-3-3) edge (m-1-2)
	    (m-3-3) edge (m-1-4)
(m-1-7) edge (m-3-6) 
(m-1-7)	edge (m-3-8)
	    (m-3-11) edge (m-1-12)
	    (m-3-13) edge (m-1-12)
	    (m-3-13) edge (m-1-14)
(m-3-16) edge (m-1-17) 
(m-3-18) edge (m-1-17);
 \end{tikzpicture}
\end{equation*}
Admissable normal forms $\W_1\subseteq \W_2 \subseteq \W_3$ associated with $K_1\subseteq K_2 \subseteq K_3$ respectively are:
\begin{align*}
\{\emptyset, 1\} &\subseteq \{\emptyset, 1, 2, 1 2 , 2 1\} \subseteq \{{\emptyset},  {1},   {3},   {1 3},    {2},   {1 2},   {2 1},   {3 2},   {2 3},   {1 2 3},   {3 2 1},   {1 3 2},   {2 1 3}\}    
\end{align*}
For $ n\geq 4$ the vertex $n-1$ is a crossing and either a sink or a source. Inductively we have admissible normal forms $\W_{n-2} \subseteq \W_{n-1}$ associated with $K_{n-2}\subseteq K_{n-1}$ respectively such that $n-1$ appears at most once in any word of $\W_{n-1}$.  On the other hand we have admissible normal forms $\{\emptyset, n\} \subseteq \{\emptyset, n, n-1, n\konk n-1 , n-1\konk n\}$ associated with the quivers $n$ and ${n-1} \rightarrow n$ (or $n-1\leftarrow n$) respectively. Now the Lemma \ref{Normalformen das Lemma Induktionsargument} with $s=n-1$ yields the admissible normal form
\[
\W_n\ldef \W_{n-2} \:\cup\: \W_{n-2}\konk n \bigcup \W_{n-1}\setminus\! \W_{n-2} \bigcup n\konk(\W_{n-1}\setminus\! \W_{n-2}) \bigcup (\W_{n-1}\setminus\! \W_{n-2})\konk n
\]
which contains $\W_{n-1}$ and fulfils the condition on the appearance of $n$. Therefore the dimension $|\W_n|$ of ${k}\monoid{K_n}\cong \B_{K_n}$ can be calculated over the recurrence relation:
\[
|\W_{n}|= 2 |W_{n-2}| + 3(|W_{n-1}| - |W_{n-2}|) = 3|W_{n-1}| -|W_{n-2}|
\]
Hence $(|\W_j|)_{j\in\NN}$ corresponds to the partial sequence $(F_{2j+1})_{j\in\NN}$ of the Fibonacci-sequence \linebreak$(F_n)_{n\in \NN}$ with $F_1=F_2=1$.

\section{The Gabriel Quiver of $\mathbf{\B_Q}$}
As before let $k$ be a field and $Q$ a finite, acyclic quiver. Theorem \ref{extkoecher-theo-Extkoecher} is a direct consequence of the next two subsections. 

One remark before we start. In \cite{DHST} the authors determined among other things the radical and Gabriel quiver of the monoid algebra of a finite $\mathcal{J}$-trivial monoid. This could be applied here, because $\B_Q$ is -- as the monoid algebra of the Hecke-Kiselman monoid $HK_Q$ -- such a monoid algebra. We repeat the argument mentioned in $\cite{MazorchukKiselman}$: Since being \mbox{$\mathcal{J}$-trivial} is closed under quotients (see \cite{Pin}), it suffices to show, that $HK_Q$ is a quotient of a $\mathcal{J}$-trivial monoid. Now the Hecke-Kiselman monoid $HK_{K_n}$ associated to the quiver $K_n$ with $n$ vertices $\{1,\ldots,n\}$ and arrows $i\rightarrow j$ for each pair $i<j$ is the Kiselman semigroup and $\mathcal{J}$-trivial by \cite{MazorchukKudryavtseva}. Moreover $Q$ can be embedded in the quiver $K_n$ for $n\ldef |Q_0|$ by choosing an enumeration $\{1,\ldots, n\}$ of the vertices $Q_0$ such that $i\rightarrow j$ implies $i<j$. Thus there is the canonical projection introduced in \cite{MazorchukKiselman} from $HK_{K_n}$ onto $HK_{Q}$.

Here, we compute the radical and the Gabriel-quiver of $\B_Q$ directly just using the defining relations. 
\subsection{The simples and the radical of $\mathbf{\B_Q}$}\label{extkoecher Abschnitt Einfache und Radikal}
The structure of the simple modules are closely related to those of the $0$-Hecke algebra (see \cite{Norton} and \cite{Hivert}). For every subset $M$ of $Q_0$ we define $\E_M=({k},\delta_M)$ \label{extkoecher-die Einfachen} to be the (one-dimensional)  $\B_Q$-module given by the homomorphism $\delta_M\colon \B_Q \longrightarrow {k}\kong{}\End{k}{k}$ of algebras with $X_q\mapsto 1$ if $q\in M$ and  $X_q\mapsto 0$ otherwise for all $q\in Q_0$.  
In the sequel we compute the radical to show that this family $(\E_M)_{M\subseteq Q_0}$ of $2^{|Q_0|}$ simple modules represents all simple $\B_Q$-modules. For this we construct for each subset $M$ of $Q_0$ a specific monomial \mbox{$X_M \in \B_Q$} to describe a $k$-linear generating set of the radical. To this end we consider the inductively defined sets of sinks $\senke{M}{j}$ at level $j\in \NN$ associated to $M$:
\[
 \begin{aligned}
\senke{M}{0} &\ldef \set{q\in M}{ q \text{ is a sink in } Q_M}\\  
\vdots&\\
\senke{M}{j+1} &\ldef \set{q\in M}{q \text{ is a sink in } Q_{M\ohne(\senke{M}{0}\cup \senke{M}{1}\cup\ldots \cup \senke{M}{j})}}
 \end{aligned}
\]
Since $M$ is finite, there is a uniquely determined index $s(M)\ldef m$ such that \mbox{$\senke{M}{m}\neq \emptyset = \senke{M}{m+1}$} holds. Moreover $Q_{\senke{M}{j}}$ contains no arrows. Hence all $X_p$ and $X_q$ with $p,q \in \senke{M}{j}$ commute and we can thus define:
\[
X_M \ldef \prod_{q\in \senke{M}{m}} \!\!\!X_q  \:\:\ldots \prod_{q\in \senke{M}{0}} \!\!\!X_q
\]
Note that by the generalized relations $X_M$ is idempotent, since we have $X_q X_M = X_M = X_M X_q$ for every $q\in M$. 
\begin{prop}\label{extkoecher prop Radikal und alle Einfachen}
The radical $\rad{\B_Q}$  of $\B_Q$ is the $k$-linear span of \mbox{$\mathcal{M}\ldef\set{X_{\{w\}}-X_w}{w\in \free{Q}}$}. So $\B_Q$ is a basic algebra and $(\E_M)_{M\subseteq Q_0}$ is a representative system of its simple modules.
\end{prop}

\begin{proof}
Let $\mathcal{V}$ be the $k$-linear span of $\mathcal{M}$ and $\mathcal{I}$ be the intersection of the annihilators \linebreak\mbox{$\Ann{\B_Q}{\E_M}\ldef\set{a\in \B_Q}{\delta_M (a) = 0}$} with $M\subseteq Q_0$. Firstly we observe that $\mathcal{V}$ coincides with the ideal $\mathcal{I}$: by the definition of $\delta_M$  we have for an arbitrary element $b=\sum_{v \in \free{Q}} b_v X_v$ in $\B_Q$: 
\[
\delta_{M}(b) = \sum\nolimits_{v \in \free{Q}, \:\:\{v\} \subseteq M} b_v
\]
Thus $b$ lies in $\mathcal{I}$ if and only if $\sum_{v \in \free{Q},\:M=\{v\}}\: b_v \:= 0$ for all   $M\subseteq Q_0$, which is in turn equivalent to 
\[
b= \sum\nolimits_{M\subseteq Q_0}\sum\nolimits_{v\in \free{Q},M= \{v\}} b_v (X_v-X_M) \in \mathcal{V}
\] 
Secondly we show, that $\mathcal{M}$ consists of nilpotent elements. (Hence $\mathcal{V}$ is a nilpotent ideal by a theorem of Wedderburn (see \cite{Pierce}, 4.6)). For this we prove the equality $X_w^{s} = X_{\{w\}}$ for each word $w$ over $Q_0$ and $s\ldef s(\{w\})$ by an induction on $s$: if all the letters occuring in  $w$  correspond to sinks, $X_w$ already coincides in $\B_Q$ with $X_{\{w\}}$. So now assume $\senke{\{w\}}{1}\neq \emptyset$. Furthermore let $v$ be the subword of $w$, which arises from $w$ by canceling all sinks $x\in \senke{\{w\}}{0}$ in $w$. Since $s(\{v\})=s-1$, it follows inductively $X_w^{s} =  X_v^{s-1}X_w= X_{\{v\}}X_w= X_{\{w\}}$ by the generalized relations and the properties of $X_{\{v\}}$ respective $X_{\{w\}}$. Therefore the element $X_{\{w\}}-X_w\in \B_Q$ is nilpotent:
\[
(X_{\{w\}}-  X_w)^{s} = (-X_{\{w\}} +X_w^2)(X_{\{w\}}-  X_w)^{s-2}=  \ldots = (-1)^{s-1}(X_{\{w\}}-  X_w^{s})=0
\]
Now  $\mathcal{V}\subseteq \rad{ \B_Q} \subseteq \mathcal{I}$ follows from the different characterizations of the radical of a finite dimensional algebra.
\end{proof}

\subsection{The Gabriel quiver of $\mathbf{\B_Q}$} \label{extkoecher Abschnitt Berechnund des Gabriel Koechers}
We calculate the $k$-dimensions of the extension groups $\Ext{\B_Q}{1}{\E_M,\E_N}$, i.e.\\ the number of arrows from $[\E_M]$ to $[\E_N]$ in the Gabriel quiver $\extko{\B_Q}$ of $\B_Q$. Then the algebra $\B_Q$ is a quotient of the path algebra $k\extko{\B}$ by an ideal $I$ with $\rad{\B_Q}^2\subseteq I \subseteq \rad{\B_Q}^r$ for some $r\in \NN$. We will see that this ideal is zero for the $m$-subspace quiver as well as for some simply shaped quivers. We devote section 5 to the proof that $I$ is generated by the commutativity relations of  $\extko{\B_Q}$, if $Q=Q_n$ is the linearly oriented Dynkin quiver of type $A$. 

Since the simple modules are one dimensional, it suffices to determine the two dimensional $\B_Q$-modules. The calculations are similar to those for the $0$-Hecke-algebras (type $A$) -- as done for example in \cite{Fayers} with the difference that we have to respect the non-symmetry of the defining relations and the generalisation to finite, acyclic quivers. 

Let $M$ and $N$ be subsets of $Q_0$. We consider the characteristic tuples $(m_q\ldef\delta_M(X_q))_{q\in Q_o}$ and $(n_q\ldef \delta_N(X_q))_{q\in Q_0}$ over $0,1 \in k$. Then we call a tuple $a=(a_q)_{q\in Q_0}$ over $k$ (or a function $a\colon Q_0\longrightarrow k, q\mapsto a_q$) admissible or $(M,N)$-admissible, if the assignment 
\[
X_{q} \mapsto \begin{pmatrix} n_q & a_q \\ 0 & m_q  \end{pmatrix}\rdef A_q \in {k}^{2\times 2}
\]
extends uniquely to an homomorphism from $\B_Q$ to $k^{2\times 2}$. We receive a two-dimensional $\B_Q$-module $W(a,M,N)=W(a)$ and up to equivalence the short exact sequences $\eta$ in $\Ext{A}{1}{\E_M,\E_N}$  are the sequences $\eta_a$:
\begin{center}
        \begin{tikzpicture}[description/.style={fill=white,inner sep=2pt}]
            \matrix(m) [matrix of math nodes, row sep=2em, column sep=2em, text height=1.5ex, text depth=0.25ex]
            {\eta_a: & 	0 & \E_N & W(a) & \E_M & 0\\};
            \path[->,font=\scriptsize]
            (m-1-2) edge (m-1-3)
            (m-1-3) edge node[above]{$\begin{pmatrix} 1\\ 0 \end{pmatrix}$ } (m-1-4)
	    (m-1-4) edge node[above]{$\begin{pmatrix} 0 & 1 \end{pmatrix}$} (m-1-5) 
	    (m-1-5) edge  (m-1-6);
\end{tikzpicture}
    \end{center}
The dimension of $\Ext{\B_Q}{1}{\E_M,\E_N}$ could be described in terms of the following property for the set-theoretic differences $M\ohne N$ and $N\ohne M$: \label{extkoecher-Differenzmenge}
\begin{df}\label{extkoecher-stark verbunden}
Let $P$ and $R$ be two disjoint subsets of $Q_0$. We say that ``$P$ is strongly connected towards $R$`` and write $P\stver R$, if neither $P$ nor $R$ are empty and for every two vertices $p\in P$ and $r \in R$ there exists (at least) one arrow from $p$ to $r$ (in $Q_1$). \footnote{In this case the subquiver of $Q$ with the vertices $P\cup R$ and just those arrows from $Q_1$, which connect vertices from $P$ to vertices from $R$, is a completely bipartite quiver.}
\end{df} 
If $P$ is not strongly connected towards $R$ we write $P\nstver R$. Since $Q$ is acyclic, $P\nstver R$ holds if and only if $P=\emptyset$, $R=\emptyset$ or there are vertices $p\in P$ and $r\in R$, such that $p$ is a sink in $Q_{p\konk r}$. 
\begin{lem}
(a)\quad We have $\Ext{\B_Q}{1}{\E_M,\E_N} = 0$ if and only if $M\ohne N \nstver N\ohne M$ holds.\newline
(b)\quad We have $\dim_{k}\Ext{\B_Q}{1}{\E_M,\E_N} = 1$ if and only if $M\ohne N \stver N\ohne M$ holds.
\end{lem}

\begin{proof}
Firstly we determine the $(M,N)$-admissible functions. Let  $a\colon Q_0\longrightarrow k, q\mapsto a_q$ be a function. The shape of the assigned $A_q$ depends on whether $q\in M\cap N$,\: $q\in Q_0\!\setminus\!(M\cup N) $,\: $q\in M \!\setminus\! N$\: or \: $q\in N\!\setminus\! M$ and is as follows
\[
A_q=\begin{pmatrix} 1 & a_q \\ 0 & 1 \end{pmatrix} \text{,}\quad A_q=\begin{pmatrix} 0 & a_q \\ 0 & 0 \end{pmatrix} \text{,}\quad A_q=\begin{pmatrix} 0 & a_q \\ 0 & 1 \end{pmatrix} \text{,}\quad A_q=\begin{pmatrix} 1 & a_q \\ 0 & 0 \end{pmatrix}
\]  
respectively. If $a$ is an admissable function, then all these matrices are idempotent, i.e.\ $a_q=0$ for all $q\in M\cap N$ and for all $q\in Q_0\!\setminus\!(M\cup N)$. Therefore every $(M,N)$-admissible function lies in
\[
\mathcal{F}\ldef \set{a\colon Q_0\rightarrow k}{a|_{M\cap N} = 0\,\text{ and }\quad a|_{Q_0\ohne (M\cup N)} = 0}  
\]
and thus it suffices to consider the remaining relations, as for example $A_pA_qA_p=A_qA_pA_q$, just for all $p,q \in M \!\setminus\! N \cup N\!\setminus\! M$. 
For example for any two elements $p,q$ in $M \!\setminus\! N$ we conclude $a_p=a_q$ from the conditions, that $A_pA_q$ equals one of the products $A_qA_pA_q$ or $A_qA_p$.                                                                                                                                                                                                                                                                                                               Similar considerations show that if $M\ohne N \stver N\ohne M$ holds then the set of all admissible functions is
\[
\set{a \in \mathcal{F}}{\exists\, c_M, c_N\in k:   \quad a|_{M\setminus N} = c_M \id_{M\setminus N}\,\text{ and }\quad a|_{N\setminus M} = c_N \id_{N\setminus M}} 
\]
Whereas in the case $M\ohne N \nstver N\ohne M$ the existence of $M\!\setminus\! N \ni q \leftarrow p \in N\!\setminus\! M$ yields the equality $A_pA_q = A_pA_qA_p$, i.e. $ a_q+a_p=0 $, hence the set of admissible functions is
\[
\set{a \in \mathcal{F}}{\exists c\in k:   \quad a|_{M\setminus N} = c \id_{M\setminus N}\quad \wedge \quad a|_{N\setminus M} = - c \id_{N\setminus M}} 
=  \langle (m_q-n_q)_{q\in Q_0} \rangle_{k}
\]
So $M\ohne N \nstver N\ohne M$ holds if and only if every $(M,N)$-admissible function lies in \mbox{$\langle (m_q-n_q)_{q\in Q_0} \rangle_{k}$}. Hence it suffices to show that $\Ext{\B_Q}{1}{\E_M,\E_N} = 0$ holds iff every $(M,N)$-admissible function lies in \mbox{$\langle (m_q-n_q)_{q\in Q_0} \rangle_{k}$}. For this let $\eta_0$ denote the trivial short exact sequence and $W(0)$ its middle term; in particular $W(0)\cong \E_M\oplus \E_N$. We show for every admissible function $a$ the equivalence:
\[
 (*)\quad \eta_a\sim \eta_0 \Longleftrightarrow a\in \langle (m_q-n_q)_{q\in Q_0} \rangle_{k}
\]
If $c\in k$ and $(a_q = c m_q- cn_q)_{q\in Q_0}$ an admissible function, then a $\B_Q$-homomorphism from $W(a)$ to $W(0)$, which additionally provides $\eta_a\sim \eta_0$,  is given by the left multiplication with $\begin{pmatrix} 1 & c \\ 0 & 1 \end{pmatrix}$. On the other hand every homomorphism $\Phi\colon W(a)\longrightarrow W(0)$ providing $\eta_a \sim \eta_0$ is given by a matrix $\begin{pmatrix} 1 & c \\ 0 & 1 \end{pmatrix}$ for a constant $c\in k$ by the commutativity of the corresponding diagram. Since $\Phi$ is in particular a $\B_Q$-homomorphism, it follows $a_p + c m_p = c n_p$ for all $p\in Q_0$. Hence (a) is proved.

Statement (b) follows from (a) if one shows that any two non-trivial short exact sequences are linearly dependent. So let $a$ and $b$ be $(M,N)$-admissible functions such that $[\eta_0]\neq [\eta_a]\neq [\eta_b]\neq [\eta_0]$ holds. By $(*)$ and the previous thoughts there exist $c\neq c'$ resp. $d\neq d'$ in $k$  with
\[
 a|_{M\setminus N} = c \id_{M\setminus N}\:\:\text{and }\:\: a|_{N\setminus M} = c' \id_{N\setminus M}\:\:\text{resp. }\:\:b|_{M\setminus N} = d \id_{M\setminus N}\:\:\text{and }\:\:b|_{N\setminus M} = d' \id_{N\setminus M}
\] 
Then $e\ldef {d+d'}/{c+c'} \neq 0$ and the left multiplication with $\begin{pmatrix} e & c'e-d' \\ 0 & 1\end{pmatrix}$ is a $\B_Q$-isomorphism from $W(a)$ to $W(b)$ providing 
 \[
\eta_a\sim e^{-1} \eta_b\colon  0\longrightarrow \E_N \overset{\begin{pmatrix} e & 0 \end{pmatrix}}{\longrightarrow} W(b) \overset{\begin{pmatrix} 0\\ 1 \end{pmatrix}}{\longrightarrow} \E_M \longrightarrow 0
\]
\end{proof}
\subsection{Properties of the Gabriel quiver}\label{subsection properties of Gabriel quiver}
Let $Q$ be a finite acyclic quiver. For this section we abbreviate notation: instead of $[\E_M]$ we just write $M$. Accordingly, the set of vertices of the Gabriel quiver is from now on just the powerset $\pot{Q_0}$ of $Q_0$ . In this notation there is (exactly) one arrow from $M\in \pot{Q_0}$ to $N\in \pot{Q_0}$ in $\extko{\B_Q}$, if $M\ohne N$ is strongly connected towards $N\ohne M$ (w.r.t. $Q$). In particular, the Gabriel quiver has no loops. Obviously, $Q$ can be embedded in $\extko{\B_Q}$. Besides $Q^{op}$ is isomorphic to the full subquiver of $\extko{\B_Q}$ with the vertices $\set{Q_0\ohne \{p\}}{p\in Q_0}$. In general, we have for all $M$ and $N$:
\[
 M \rightarrow N \in \extko{\B_Q} \Longleftrightarrow Q_0\ohne N \rightarrow Q_0\ohne M \in \extko{\B_Q}  
\]
Therefore the map $\pot{Q_0}\longrightarrow \pot{Q_0},\; M\mapsto Q_0\ohne M$ induces an involution $\iota$ on $\extko{\B_Q}$. In particular we have in the case that $\B_Q$ is already isomorphic to $k\extko{\B_Q}$: $\B_Q\kong{} \B_{Q}^{op} \kong{} \B_{Q^{op}}$.
Now we look at special cases and some examples and end this chapter with general observations.
\begin{itemize}
\item The Gabriel quiver of $\B_{Q_n}$ is described by the following equivalence:
\begin{center}
\textit{There is one arrow $M\rightarrow N$ in $\extko{\B_{Q_n}}$ if and only if there exists exactly one index $i\in \ito{n}$ with $M\ohne N =\{i\}$ and  $N\ohne M = \{i+1\}$.}
\end{center}
Therefore only equally large sets are connected in $\extko{\B_{Q_n}}$. Thus the Gabriel quiver of $\B_{Q_n}$ has at least $n+1$ connected components. 
Actually $\extko{\B_Q}$ has exactly $n+1$ as we will see. 
\item In the $m$-subspace quiver $T_m$ each subset of sources is strongly connected towards the only sink set $\{s\}$. Thus we have:
\begin{center}
\textit{There is one arrow $M\rightarrow N$ in $\extko{\B_Q}$ if and only if there exist two disjoint subsets $M'\neq \emptyset $ and $N'$ of $\ito{m}$ with $M = N'\cup M'$ and $N=N' \cup \{s\}$.}
\end{center} 
As in the first special case $\emptyset$ and $(T_m)_0$ are isolated vertices of the Gabriel quiver. Moreover each vertex $\emptyset\neq M\neq (T_m)_0\in\pot{Q_0}$ is either a source (if $M\subseteq \ito{m}$) or a sink (if $s\in M$). Hence there are no paths of length $\geq 2$. Consequently the algebra $k\extko{\B_Q}$ has radical square $0$ and is already isomorphic to $\B_Q$. In particular $\B_Q \cong \B_Q^{op}$. 
By the way the only arrow in $\extko{\B_Q}$ which is fixed under the above presented involution $\iota$ is $\ito{n} \rightarrow \{s\}$. 
\end{itemize}
We assume $Q$ to be connected. We call $q\in Q_0$ a successor of $p\in Q_0$, if there is a path in $Q$ from $p$ to $q$. By Theorem \ref{extkoecher-theo-Extkoecher} it follows straightforwardly that the sinks (sources) of the Gabriel quiver of $\B_Q$ are precisely the subsets of $Q_0$ which are closed under successors  (predecessor) in $Q$. \label{Gabriel quiver remark sources and sinks} Therefore we now know the projective (injectives) amongst the simple $\B_Q$-modules. 
We end this section with the proof of Proposition \ref{Gabriel prop connected components}. Successively applying the next lemma shows that the full subquiver $\extko{\B_Q}_j$ of $\extko{\B_Q}$ whose vertices are those subsets of $Q_0$ with exactly $j$ elements is connected for every $j\leq |Q_0|$: 
\begin{lem}
For every proper non-empty subset $A$ of $Q_0$, each $a\in A$ and each $b\in Q_0\ohne A$ there exists a walk between $A$ and \mbox{$A\ohne\{a\}\cup \{b\}$} in $\extko{\B_Q}$. 
\end{lem}
\begin{proof}
In the sequel we write $x\wa y$, if the vertices $x$ and $y$ are connected by an arrow. Since $Q$ is connected, there is a walk between any two vertices $a$ and $b$ in $Q_0$:
\[
a = x_0 \wa x_1 \wa  \ldots \wa x_r \wa  x_{r+1} = b 
\]
Without loss of generality we can assume, that $x_0, \ldots, x_{r+1}$ are pairwise disjoint. Based on such a walk we construct inductively on the number of the changes from $A$ to $Q_0\ohne A$, i.e.\ on the number $n$ of indices $j\in \oto{r+1}$ with $x_j\in A$ and $x_{j+1}\in Q_0\ohne A$, a walk in $\extko{\B_Q}$ between $A$ and \mbox{$A\ohne\{a\}\cup \{b\}$}.
\newline\case{$n=1$:} Let $a\in A\subseteq Q_0$ and $b\in Q_0\ohne A$, such that there is a walk in $Q$ of the following kind: 
\begin{align*}
&\qquad\underbrace{a\wa x_1 \wa \ldots \wa x_{j-1}}_{\in A} \wa \underbrace{x_j \wa \ldots \wa x_{r}\wa b}_{\in Q_0\ohne A} 
\shortintertext{Then we obtain the following walk in the Gabriel quiver by Theorem \ref{extkoecher-theo-Extkoecher}:}
A &\wa  A\ohne \{x_{j-1}\}\cup \{x_j\} \wa A\ohne \{x_{j-1}\}\cup \{x_{j+1}\}\wa \ldots 		
\wa  A\ohne \{x_{j-1}\}\cup \{x_r\} \wa A\ohne \{x_{j-1}\}\cup \{b\} \\
& \wa A\ohne \{x_{j-2}\}\cup \{b\} \wa \ldots \wa A\ohne \{x_{1}\}\cup \{b\}\wa A\ohne \{a\}\cup \{b\}
\end{align*}
\case{$n\rightarrow n+1$:} Now let $a\in A\subseteq Q_0$ und $b\in Q_0\ohne A$ such that, there is a walk in $Q$ over pairwise disjoint vertices between $a$ and $b$ with more than one change between $A$ and $Q_0\ohne A$. Such a walk is of the kind:
 \[
\begin{aligned}
\underbrace{a= x_0\wa\ldots \wa x_{j-1}}_{\in A} &\wa \underbrace{x_j \wa \ldots \wa x_{l-1}}_{\in Q_0\ohne A} 
\wa \underbrace{x_l}_{\in A} \wa \ldots \wa \underbrace{x_{r+1} = b }_{\in Q_0\ohne A}
\end{aligned}
\]
with $j-1\geq 0$ and $l-1\geq j$. By the induction hypothesis there exists in $\extko{\B_Q}$ a walk between $A$ and \mbox{$A\ohne \{x_l\} \cup \{b\} \rdef B$}. Since all $x_0,\ldots,x_{r+1}$ are pairwise disjoint the following walk is a walk with just one change (now from $B$ to $Q_0\ohne B$): 
 \[
\begin{aligned}
\underbrace{a\wa x_1 \wa \ldots \wa x_{j-1}}_{\in B} &\wa \underbrace{x_j \wa \ldots \wa x_{l-1} \wa x_l}_{\in Q_0\ohne B}
\end{aligned}
\]
Again by the induction hypothesis there is a walk in the Gabriel quiver between $B$ and $B\ohne\{a\} \cup \{x_l\} = A\ohne\{a\} \cup \{b\}$ which finishes our walk from $A$ to $A\ohne\{a\} \cup \{b\}$.
\end{proof}
\begin{proof}[Proof of Proposition \ref{Gabriel prop connected components}.]
As we observed above, the Gabriel quiver of $\B_{Q_n}$ has at least $n+1$ connected components. Assertion (a) follows from the last lemma. Now let $Q$ be distinct from any $Q_n$. Therefore $Q$ has a subquiver of the form $x_1\rightarrow s\leftarrow x_2$ or  $x_1\leftarrow s \rightarrow x_2$. In each case the subsets  $\{x_1,x_2\}$ and $\{s\}$ are strongly connected. Consequently there is an arrow in the Gabriel quiver between $\{x_1,x_2\}$ and $\{s\}$. Thus $\extko{\B_Q}_2$ and $\extko{\B_Q}_1$ are connected. The involution $\iota$ yields an arrow in $\extko{\B_Q}$ between the subsets $Q_0\ohne\{x_1,x_2\}$ and $Q_0\ohne\{s\}$, which connects $\extko{\B_Q}_{n-2}$ and $\extko{\B_Q}_{n-1}$. Now for each subset $D\subseteq Q_0\ohne\{x_1,x_2,s\}$ there is an arrow in $\extko{\B_Q}$ between $D\cup\{x_1,x_2\}$ and $D\cup\{s\}$ linking the subquivers $\extko{\B_Q}_{|D|+1}$ and $\extko{\B_Q}_{|D|+2}$. Hence $\extko{\B_Q}_1, \ldots, \extko{\B_Q}_{n-1}$ are connected. Meanwhile $\extko{\B_Q}_0=\bullet_{\emptyset}$ and $\extko{\B_Q}_{|Q_0|}=\bullet_{Q_0}$.
\end{proof}

\section{The monoid algebra attached to $\mathbf{Q_n}$}
In this section let $R$ be a field. Recall that $\Po_n$ denotes the poset of the product order $\grn{n}$ defined in the introduction.  We first introduce the elements of $\A{n}\ldef R\monoid{Q_n}$ to state the precise isomorphism between the incidence algebra $\Inc{\Po_n}$ and $\A{n}$. Then we prove the stated properties in the subsections 5.1, 5.2 and 5.3. 

As is well-known, the incidence algebra $\Inc{\Po_n}$ is the free $R$-space over $X_{(J,I)}$ with $J\grn{n} I$ endowed with the multiplication given by $X_{(K,J)}X_{(J',I)} = X_{(K,I)}$ if $J=J'$ and $X_{(K,J)}X_{(J',I)}=0$ otherwise. It is the path algebra of the Hasse diagramm of the poset $\Po_n$ modulo the ideal which is generated by the commutativity relations (e.g.\ see \cite{Ringel}).   

As we have already seen $\A{1}\subseteq \A{2}\subseteq \A{3}\subseteq \ldots$ is a tower of algebras. Thus an inductive description of elements is possible. Local properties such as idempotency or orthogonality are preserved, in contrast to global ones such as centrality or being the unit element $1$ in $\A{n}$, when viewing elements in a greater algebra $\A{n+k}$.

The heart of the definition of the idempotents in $\A{n}$ corresponding to the idempotents $X_{(J,J)}$ in $\Inc{\Po_n}$ are the inductively (on $n$) defined elements $\y{1}{n}, \y{2}{n},\ldots ,\y{n}{n}$ in $\A{n}$. We start with $\y{1}{1}\ldef \p{1}$ and set for all $k\in \ito{n}$:  
\begin{align*}
\y{k}{n} \ldef\begin{cases}
                 \y{1}{n-1}-\p{n}\y{1}{n-1}+ \p{n} \quad &\text{if } k=1 \\[3pt]
                 \y{k}{n-1}-\p{n}\y{k}{n-1}+ \y{k-1}{n-1}\p{n} \quad &\text{if } 2\leq k \leq n \\[1pt]
                 0 \quad &\text{if } k>n \\   
                \end{cases}
\end{align*}
Each of these elements generates an ideal which is closely related to the admissible normal form given in \ref{Special cases  prop normal form of B_n lin.oriented} (see Corollary \ref{Dynkin Cor Ideal generators}). Their properties are listed in Lemma \ref{Idempotente Lemma Eigenschaften der y}. From them we conclude that the idempotents in $\A{n}$ corresponding to the connected components of $\Inc{\Po_n}$ are given by:  
\[
 \z{0}{n} \ldef 1-\y{1}{n}, \quad\ldots,\quad  \z{k}{n} \ldef \y{k}{n} - \y{k+1}{n},\quad \ldots, \quad\z{n}{n}\ldef \y{n}{n} - \y{n+1}{n}=\y{n}{n}
\]
Furthermore we consider the inductively defined elements $\g{J}{n}\in \A{n}$ for all $\emptyset\neq J \subseteq \ito{n}$ starting with $\g{\{1\}}{1} \ldef \p{1}$:
\[
\g{J}{n} \ldef \begin{cases}
	      \g{J}{n-1}-\p{n}\g{J}{n-1} \quad &\text{if } n\notin J\\
	      \p{n} \quad &\text{if } \{n\}= J\\
	      \g{J\ohne \{n\}}{n-1}\p{n} \quad &\text{if } n \in J\neq\{n\}\\ 
	       \end{cases}
\]
\begin{theo} \label{Dynkin complete pw orthogonal idempotents}
A complete system of pairwise orthogonal idempotents of $\A{n}$ consists of
\[
 \f{\emptyset}{n}\ldef\z{0}{n}\quad\text{and}\quad  \f{J}{n} \ldef \g{J}{n}\z{|J|}{n}\quad \text{with} \quad \emptyset\neq J\subseteq \ito{n}  
\]
\end{theo}
(In \cite{DHST} a complete set of orthogonal idempotents for $\A{n}$ is given explicitly.) Now we turn towards those elements in $\A{n}$ corresponding to $X_{(J,I)}\in \Inc{\Po_n}$ with $J\neq I$. For this we define for all $n\geq j\geq i\geq 1$ the word:
\begin{align*}
 j\ru i \ldef j\konk (j-1)\konk\ldots \konk (i+1)\konk i \quad\in \ito{n}^*\\
 \shortintertext{and for all subsets $J=j_r>\ldots >j_1$ and $K= k_r>\ldots>k_1$ of $\ito{n}$ with $K\grn{n} J$ the monomial}
\p{(K,J)}\ldef \p{k_1 \ru j_1 \konk \ldots \konk k_r\ru j_r} \quad \in \A{n}\label{definition p(j,K)}
\end{align*}
The already introduced idempotent $\p{J}= \p{j_1\konk j_2 \konk \ldots \konk j_r}$  coincides with $\p{(J,J)}$; meanwhile $\p{(\emptyset,\emptyset)}= 1$ in $\A{n}$. The main difficulty is to show that the elements $\f{K}{n}\p{(K,J)}\f{J}{n}$ are distinct from $0$. 
\begin{theo}\label{Dynkin theorem aJI neq 0}
For every pair $K\grn{n} J$ we have: $ \f{K}{n}\p{(K,J)}\f{J}{n} \neq 0$.
\end{theo}
Here we will need two more descriptions of the middle factor, including one inductive one (see subsection 5.2.2) and an inductive description of the elements $\g{J}{n}$ (see Lemma \ref{Dynkin lemma inductive gJn}). Along the way we get all we need to prove the main theorem:
\begin{main}\label{Dynkin maintheorem}
Let $R$ be a field and $n\in \NN$. The $R$-linear map $\Phi\colon \Inc{\Po_n} \longrightarrow \A{n} $ with \mbox{$\Phi(X_{(K,J)})= \f{K}{n}\p{(K,J)}\f{J}{n}$} for all $K\grn{n} J$ is an  isomorphism of algebras.  
\end{main}

One remark on the field $R$: in our proof we distinguish between two monomials $\p{v}$ and $\p{w}$ by comparing their action on the injective indecomposable $Q_n$ representations $I_0, I_1\ldots I_n$ over $R$. These steps can be replaced by comparing the non-decreasing parking functions $\pi_v$ and $\pi_w$ as functions on $1,\ldots, n+1$ (notation as in \cite{Hivert}). Since we just use the defining relations for $\monoid{Q_n}$ (see Proposition \ref{Special cases  prop normal form of B_n lin.oriented}) which are the same as for $\mathrm{NDPF}_{n+1}$ (see \cite{Hivert}), we could replace the field $R$ by an arbitrary commutative ring, as mentioned in \cite{Hivert}) and \cite{DHST}:
\begin{rem}
Let $R$ be a commutative ring. Then the $R$-algebras $R\mathrm{NDPF}_{n+1}$ and $\Inc{\Po_n}$ are isomorphic.
\end{rem}

\subsection{Proof of Theorem \ref{Dynkin complete pw orthogonal idempotents}}
We start with the properties of the elements $\y{1}{n},\ldots,\y{n}{n}$ and the direct conclusions for the elements  $\z{0}{n},\ldots,\z{n}{n}$:
\begin{lem}\label{Idempotente Lemma Eigenschaften der y}
The elements $\y{j}{n}$ are central in $\A{n}$ for all $j\in \ito{n}$. Thus the elements $\z{j}{n}$ are central. Meanwhile the idempotency and orthogonality of $\z{0}{n},\ldots,\z{n}{n}$ follows directly from:
\begin{align*}
\y{i}{n}\,\y{j}{n}&=\y{j}{n} \qquad \text{for all} \quad 0<i\leq j\leq n\\
\shortintertext{Each $\y{j}{n}$ is distinct from $0$ since for all subsets $J\subseteq\ito{n}$ with $|J|=j$ the following equation holds:}
\y{j}{n}\p{J}  &= \p{J}\\ 
\end{align*}
\end{lem}
\begin{proof}
Exemplary in more detail, we prove by induction on $n$ that $\y{i}{n}$ is central for each $i\in \ito{n}$. 
For $n\in \{1,2\}$ the centrality of $\y{1}{1}=\p{1}$ and of  $\y{1}{2}=\p{1}+\p{2}-\p{2\konk 1}$ or $\y{2}{2}=\p{1\konk 2}$ in $\A{1}$ resp. in $\A{2}$ are direct consequences of the defining relations. So let $n>2$. For a generator $\p{j}$ of $\A{n}$ we have:
\[
 \p{j}\y{1}{n} -\y{1}{n}\p{j} =(\p{j}\y{1}{n-1}- \y{1}{n-1}\p{j}) + (-\p{j}\p{n}\y{1}{n-1}+\p{n}\y{1}{n-1}\p{j})+ (\p{j}\p{n} - \p{n}\p{j})
\]
For $j<n-1$ this adds up to $0$ by the induction hypothesis and the (commutativity) relations. We consider the cases $j\in \{n-1,n\}$ separately using the induction hypothesis for $j=n-1$:
\begin{align*}
 \p{n-1}\y{1}{n} -\y{1}{n}\p{n-1} 
   &= (-\p{n-1}\p{n} + \p{n}\p{n-1})\y{1}{n-1} + (\p{n-1}\p{n} - \p{n}\p{n-1})\\[2pt]
   &= (-\p{n-1}\p{n} + \p{n}\p{n-1})\y{1}{n-2} - (-\p{n-1}\p{n} + \p{n}\p{n-1})\p{n-1}\y{1}{n-2} \\[2pt]
   &\qquad  + (-\p{n-1}\p{n} + \p{n}\p{n-1})\p{n-1}  + (\p{n-1}\p{n} - \p{n}\p{n-1})\\[2pt]
   &= (-\p{n-1}\p{n} + \p{n}\p{n-1})\y{1}{n-2} - (-\p{n-1}\p{n} + \p{n}\p{n-1})\y{1}{n-2} \\ 
   &= 0
\shortintertext{and the generalized relations for $j=n$:}
 \p{n}\y{1}{n} -\y{1}{n}\p{n} 
    &= (\p{n}\y{1}{n-1}- \y{1}{n-1}\p{n}) + (-\p{n}\y{1}{n-1}+\y{1}{n-1}\p{n}) =0
 \end{align*}
Now we consider the elements $\y{2}{n},\ldots, \y{n}{n}$. For any $k > 1$ and a generator $\p{j}$ of $\A{n}$ we have: 
\[
\begin{aligned}
 \p{j}\y{k}{n} -\y{k}{n}\p{j} 
 &=(\p{j}\y{k}{n-1}- \y{k}{n-1}\p{j}) + (-\p{j}\p{n}\y{k}{n-1}+\p{n}\y{k}{n-1}\p{j})+ (\p{j}\y{k-1}{n-1}\p{n} - \y{k-1}{n-1}\p{n}\p{j})
 \end{aligned}
\]
Again for $j<n-1$, this adds up to $0$. Meanwhile we use the induction hypothesis and the generalized relations  for $j=n-1$:
\begin{align*}
\p{n-1}\y{k}{n} -\y{k}{n}\p{n-1}
&= (-\p{n-1}\p{n}+\p{n}\p{n-1})(\y{k}{n-2}-\p{n-1}\y{k}{n-2}+ \y{k-1}{n-2}\p{n-1}) \\[2pt]
&\qquad + (\y{k-1}{n-2}-\p{n-1}\y{k-1}{n-2}+ \y{k-2}{n-2}\p{n-1})(\p{n-1}\p{n} - \p{n}\p{n-1})\\[2pt] 
&= (-\p{n-1}\p{n}+\p{n}\p{n-1})\y{k-1}{n-2}\p{n-1} + (\y{k-1}{n-2}-\p{n-1}\y{k-1}{n-2})(\p{n-1}\p{n} - \p{n}\p{n-1})\\[2pt] 
&= -\p{n-1}\p{n}\y{k-1}{n-2}\p{n-1} + \p{n}\y{k-1}{n-2}\p{n-1} + \y{k-1}{n-2}\p{n-1}\p{n} - \p{n}\y{k-1}{n-2}\p{n-1} \\[2pt]
  &\qquad -\y{k-1}{n-2}\p{n-1}\p{n}  + \p{n-1}\p{n}\y{k-1}{n-2}\p{n-1}\\
&= 0\\   
\shortintertext{For $j=n$ the calculation is again simply:}
 \p{n}\y{k}{n} -\y{k}{n}\p{n} &=(\p{n}\y{k}{n-1}- \y{k}{n-1}\p{n}) + (-\p{n}\y{k}{n-1}+\y{k}{n-1}\p{n})+ (\y{k-1}{n-1}\p{n} - \y{k-1}{n-1}\p{n}) =0
 \end{align*}
Now we prove by induction that $\y{j}{n}\y{k}{n}= \y{k}{n}$ holds for all \mbox{$1\leq j\leq k\leq n$}. The case for $n=1$ is trivial, so we proceed with $n>1$. For this we set $\y{0}{n-1}\ldef 1\in \A{n}$ and rewrite:
\[
\y{j}{n}=\y{j}{n}-\p{n}\y{j}{n-1} + \y{j-1}{n-1}\p{n}=(1-\p{n})\y{j}{n} + \y{j-1}{n-1}\p{n}\in \A{n}
\]
By the generalized relations we get the following equalities in $\A{n}$; the last one by the induction hypothesis:
\[
 \begin{aligned}
\y{j}{n}\y{k}{n}
&= (1-\p{n})\y{j}{n-1} (1-\p{n})\y{k}{n-1} + (1-\p{n})\y{j}{n-1} \y{k-1}{n-1}\p{n} \\[2pt]
&\qquad \:+ \y{j-1}{n-1}\p{n}(1-\p{n})\y{k}{n-1} + \y{j-1}{n-1}\p{n}\y{k-1}{n-1}\p{n}\\[3pt]
&= (1-\p{n})\y{j}{n-1} (1-\p{n})\y{k}{n-1} + \y{j-1}{n-1}\y{k-1}{n-1}\p{n}\\[2pt]    
&= (1-\p{n})\y{k}{n-1} + \y{k-1}{n-1}\p{n}\\[2pt]
 \end{aligned}
\]
For the last statement it is convenient to show by induction on $n$ (simultanously) that for all $\emptyset\neq J\subseteq \ito{n}$ the following two equations hold:
\[
\y{|J|}{n}\p{J}  = \p{J} \qquad\text{and }\qquad \y{|J|}{n}\p{n+1}\p{J}  = \p{n+1}\p{J}  
\] 
\end{proof}
\begin{proof}[Proof of Theorem \ref{Dynkin complete pw orthogonal idempotents}.]
By similar computations and case by case analysis we deduce the \linebreak Theorem with the following steps: first show by induction on $n$, that for each $k\in \ito{n}$ the set $\set{\g{J}{n}}{J\subseteq \ito{n},|J|=k}$ is a set of pairwise orthogonal idempotents, whose elements add up to $\y{k}{n}$. Therefore -- by definition of $\f{J}{n}$ and the properties of $\z{k}{n}$ -- the set $\set{\f{J}{n}}{J\subseteq \ito{n},|J|=k}$ consists of pairwise orthogonal idempotents. From $\y{k}{n}\z{k}{n}= \z{k}{n}$ we get $\sum_{J\subseteq \ito{n},|J|=k}\f{J}{n}=\z{k}{n}$. Since $\set{\z{k}{n}}{0\leq k \leq n}$ consists of central pairwise orthogonal elements, which add up to $1\in \A{n}$, Theorem \ref{Dynkin complete pw orthogonal idempotents} follows.  
\end{proof}
We finish this subsection with a useful remark on the following chain of ideals in $\A{n}$:
 \[
\A{n}= \Jk{0}{n} \supset \Jk{1}{n} \supset \Jk{2}{n} \supset \ldots \supset \Jk{n-1}{n} \supset \Jk{n}{n} = \langle \p{\ito{n}}\rangle_R \supset 0   
 \]
where $\Jk{k}{n}$ is the ideal generated by the monomials $\p{J}$ with $J\subseteq \ito{n}$ and $|J|=k$. It is directly deduced from the last equation of Lemma \ref{Idempotente Lemma Eigenschaften der y}, that $\Jk{k}{n}$ is contained in the ideal $\y{k}{n}\A{n}$. On the other hand an easy induction shows that $\y{k}{n}$ is contained in $\Jk{k}{n}$. Thus the equalities hold: 
\begin{cor}\label{Dynkin Cor Ideal generators}
\[
 \Jk{k}{n} = \y{k}{n}\A{n} = \set{a\in \A{n}}{\y{k}{n}a= a}
\]
In particular $\z{k}{n}=\y{k}{n}-\y{k+1}{n}$ annihilates the ideals $\Jk{k+1}{n} \supset \ldots \supset \Jk{n}{n}$.
\end{cor} 
\subsection{Proof of Theorem \ref{Dynkin theorem aJI neq 0}}
\subsubsection{Chain description of $\mathbf{\p{K,J}}$}\label{chain description}
For $n\in \NN$ and subsets $J$ and $K$ of $\ito{n}$ we write $K\vgrn{n} J$ if $K$ is a minimal proper successor of $J$ w.r.t. $\grn{n}$ and say that their are neighbours. In general there is more than one $\vgrn{n}$-chain between $K\grn{n} J$. But the monomial $\p{(K,J)}$ is an invariant of such $\vgrn{n}$-chains:
\begin{lem} \label{Idempotente Lemma alternative Beschreibung der mnJK}
Let $K\grn{n} J$. For every $\vgrn{n}$-chain \mbox{$K=H_t \vgrn{n} \ldots \vgrn{n} H_1 = J$} we have:
\[
\p{(K,J)} = \p{{H_t}\konk\ldots\konk{H_1}} 
\]
\end{lem}
\begin{proof}
We just show the induction step $ 2\leq n \rightarrow n+1$ for non-empty, proper subsets $n+1 \in K\grn{n+1} J$ of $\ito{n+1}$. Let $1<r\ldef |J| \neq J$ and consider a $\vgrn{n+1}$-chain $K=H_t\vgrn{n+1} \ldots \vgrn{n+1} H_1 = J$. 
Let $ h_{r,s}>h_{r-1,s}>\ldots> h_{1,s}$ be the elements of $H_s$ for $s\in \ito{t}$. Since $H_{s+1}\vgrn{n+1} H_{s}$ there is exactly one index $k\in \ito{r}$ with $h_{k,s+1} > h_{k,s}=h_{k,s+1} -1$ and $h_{j,s}=h_{j,s+1}$ for all $j\neq k$. 

We denote by $\widetilde{J}$, $\widetilde{K}$ and $\widetilde{H_s}$ the sets $J,K$ and $H_s$ without their maximal elements respectively.
\newline\case{First case: ${n+1 \in J}$}
In particular $k_r \ru j_r = n+1$ holds. Moreover the maximal elements $h_{r,s} = n+1$ for $s\in \ito{t}$ are not involved in the $\vgrn{n+1}$-chain, that is we already have a $\vgrn{n}$-chain:
\begin{align*}
&\hphantom{\p{{H_t}\konk\ldots\konk{H_1}}}\widetilde{K} = K\ohne \{n+1\}= \widetilde{H_t}\vgrn{n}  \ldots \vgrn{n} \widetilde{H_1} = J\ohne \{n+1\}=\widetilde{J}\\
\shortintertext{So by the generalized relations and the induction hypothesis we deduce:}
\p{{H_t}\konk\ldots\konk{H_1}} &=\p{\widetilde{H_t}\konk n+1\konk\ldots\konk  {\widetilde{H_2}}\konk n+1\konk{\widetilde{H_1}}\konk n+1}= \p{{\widetilde{H_t}}\konk\ldots \konk  {\widetilde{H_2}}\konk{\widetilde{H_1}}\konk n+1}
=\p{k_1 \ru j_1 \konk \ldots \konk k_{r-1}\ru j_{r-1}\konk k_r\ru j_r} = \p{(K,J)}
\end{align*}
\case{Second case: ${n+1 \notin J}$}
We divide the chain into two chains, such that one of them contains no $n+1$ but the other does. More precisely, there is an index $s\in \{2,\ldots,t\}$ minimal with $n+1\in H_s$. Then we have: $h_{r,s-1}=n <n+1 = h_{r,s} = \ldots = h_{r,t} = k_r$. 
As in the first case we receive:
\[\p{H_t\ldots H_s}= \p{(\widetilde{H_s},\widetilde{H_t})}\p{n+1} \quad\text{and}\quad \p{H_{s-1}\ldots H_1}=\p{(\widetilde{H_{s-1}},\widetilde{H_{1}})}\p{ h_{r,s-1}\ru h_{r,1}}
\]
Now $\p{n+1}$ commutes with all $\p{j}$ for  $j\leq h_{r-1,s-1}$, hence with $\p{(\widetilde{H_{s-1}},\widetilde{H_1})}$. Moreover we observe $\widetilde{H}_{s-1} = \widetilde{H}_s$. Therefore we get by applying the induction hypothesis several times:
\[
 \begin{aligned}
\p{H_t\ldots H_1} = \p{(\widetilde{H_t},\widetilde{H_s})}\p{(\widetilde{H_{s-1}},\widetilde{H_{1}})}\p{n+1\konk h_{r,s-1}\ru h_{r,1}}=\p{(\widetilde{H_t},\widetilde{H_1})}\p{n+1\ru j_r}=\p{(K,J)}\\
 \end{aligned}
\]
\end{proof}
\begin{cor}\label{Idempotente Korollar pJK in den Idealen und invariant unter PJ und PK}
For all $L\grn{n}K\grn{n}J$ the monomial $\p{(K,J)}$ is contained in the ideal $\Jk{|J|}{n}$ and we have $\p{(L,K)}\p{(K,J)}=\p{(L,J)}$, in particular \mbox{$\p{K}\p{(K,J)} = \p{(K,J)} = \p{(K,J)}\p{J}$} holds.
\end{cor}
In fact, the stronger assertion holds:
\begin{lem}\label{Dynkin lemma pKJ not in Jk k+1}
\[   
\p{(K,J)} \in \Jk{|J|}{n}\ohne \Jk{|J|+1}{n}
\]
\end{lem}
\begin{proof}
The monomials $\p{(K,J)}$ for maximal $K$ and minimal $J$ (w.r.t. $\kln{n}$) are easier to handle, i.e.\ $K=\{n-r+1, \ldots, n\}$ and $J=\ito{r}$ for some $r\in \oto{n}$. We set $\m{0}{n}\ldef 1\in \A{n}$ and for each $r \in \ito{n}$
\[
 \m{r}{n}\ldef \p{(\{n-r+1, \ldots, n\},\ito{r})} =\p{n-r+1\ru 1\konk \ldots \konk n-1\ru r-1\konk n\ru r} =\m{r-1}{n-1}\p{n\ru r} 
\]
Now we prove $\m{r}{n} \notin \Jk{r+1}{n}$ by induction on $n$, which is by Corollary \ref{Dynkin Cor Ideal generators} equivalent to: 
\[
  (1-\y{r+1}{n}) \m{r}{n}\neq 0
\]
In the induction step $n\rightarrow n+1$ the extreme cases $r\in \{0,n+1\}$ are trivial and for the remaining $r$ we get by direct calculations using the generalized relations:
\[
(1-\y{2}{n+1})\m{1}{n+1}= \z{0}{n}\p{n+1\ru 1} \quad\text{and}\quad (1-\y{r+1}{n+1})\m{r}{n+1} =  (1-\y{r}{n}) \m{r-1}{n}\p{n+1\ru r}
\]
The induction hypothesis applies to the ladder cases. Thus the linear combinations $\z{0}{n}$ and $(1-\y{r}{n}) \m{r-1}{n}$ of monomials in $\A{n}$ are distinct from $0$. But for any two distinct monomials $\p{v}$ and $\p{w}$ in $\A{n}$ there exists an injective indecomposable $Q_n$-representation $I_y$ with $\p{v}I_y\neq \p{w}I_y$ by the proof of Proposition \ref{Special cases  prop normal form of B_n lin.oriented}.
With Lemma \ref{Special_cases lem embedding functor} we conclude that thus $\p{v}\p{n+1\ru r}$ and $\p{w}\p{n+1\ru r}$ act differently on one of the $Q_{n+1}$-representations $I_y$ or $I_{y+1}$. Therefore the monomials $ \m{r}{n+1}$ do not lie in $\Jk{r+1}{n+1}$. Since each monomial $\p{(K,J)}$ is a factor of $\m{|J|}{n}$ by the chain description, $\p{(K,J)}$  does not lie in $\Jk{|J|+1}{n}$. 
\end{proof}

\subsubsection{Inductive description of $\mathbf{\p{(K,J)}}$}\label{inductive description}
We denote by $K_{\max}$ the greatest interval (w.r.t. $\groint$, see page \pageref{specialcases_groint})  of a finite subset $K$ of $\NN$, e.g.\ $\{9,8,5,4,2\}_{\max}=\{9,8\}$ and $\{4,3,2\}_{\max}=\{4,3,2\}$ and $\{5,3,2,1\}_{\max}=\{5\}$. As usual we set $K-1\ldef \set{k-1}{k\in K}$ and $\emptyset - 1\ldef \emptyset$.  
If $K\grn{n+1} J$ and $n+1$ lies in $J$, then $n+1$ also lies in $K$, moreover $J_{\max}$ is a subset of $K_{\max}$. Therefore the distinction of cases in the next remark is complete.
\begin{rem}
Let $K\grn{n+1} J$ and $K\neq J$. Then we have:
\[
 \begin{aligned}
&K\grn{n} J \quad &&\text{if } n+1\notin K\\[2pt]
&K\ohne K_{\max}\cup (K_{\max}-1)\grn{n} J \quad &&\text{if }n+1\in K\ohne J\\[2pt]
& K\ohne K_{\max}\cup ((K_{\max}\ohne J_{\max})-1)\grn{n} J\ohne J_{\max} \quad &&\text{if }n+1\in  J  
 \end{aligned}
\]
\end{rem}
Thus we can define:
\begin{df} \label{Idempotente zu den Wegen korrespond. Elemente in An}
Starting with $\m{\{1\},\{1\}}{1} \ldef \p{1}$ and $\m{\emptyset,\emptyset}{1} = 1\in \A{1}$ we define inductively on $n$ for every pair $K\grn{n+1} J$ the monomial $\m{K,J}{n+1}$ by:
\[
\m{K,J}{n+1}\ldef \begin{cases}
\p{J}  \quad &\text{if } J=K\\
\m{K,J}{n}\quad 			&\text{if } J\neq K , n+1\notin K\\[3pt]
\p{K_{\max}}\:\m{K\ohne K_{\max}\cup (K_{\max}-1)\;,\;J}{n}\quad &\text{if } J\neq K ,  n+1\in K\ohne J \\[3pt]	 
\p{K_{\max}}\:\m{K\ohne K_{\max}\cup ((K_{\max}\ohne J_{\max})-1)\;,\;J\ohne J_{\max}}{n} \quad &\text{if } J\neq K, n+1\in J 
\end{cases}
 \] 
\end{df}
An induction on $n$ and a case by case analysis according to the definition of $\m{K,J}{n}$ shows in a straightforward way:
\begin{lem} \label{Idempotente Lemma induktive Beschreibung der mnJK}
Let $K\grn{n} J$. Then we have:
\[
 \p{(K,J)} = \m{K,J}{n}
\]
\qed\end{lem}
Before we use the inductive description of $\p{(K,J)}$ we need to introduce two more notation. With them we can formulate an inductive description of $\g{J}{n}$ and hence of $\f{J}{n}$ (see Lemma \ref{Dynkin lemma inductive gJn}).  
\begin{df}
For a subset $N$ of $\ito{n}$ we define the element $\yJ{N}$ in $\A{n}$ by
\[
 \yJ{N} = \begin{cases}
 0 \quad &\text{if } N = \emptyset\\
\yJ{N\ohne\max N} -\p{\max N}\yJ{N\ohne \max N} + \p{\max N}\quad &\text{if } N\neq \emptyset           
          \end{cases}
\]
\end{df}
So $\yJ{N}$ is similarly defined to $\y{1}{n}$ and has similar properties (w.r.t. to the subalgebra $\A{N}$ of $\A{n}$ generated by $\p{s}$ with $s\in N$), namely: the element $\yJ{N}$ is central in $\A{N}$. Moreover, for all  $m\in N$ and each $x>\max N$ we have the identities $\yJ{N}\p{m}  = \p{m}$ and $\yJ{N}\p{x}\p{m} = \p{x}\p{m}$ and consequently $\yJ{N}\neq 0$.
We consider these elements for the sets:
\begin{df}
 We define for each subset $K$ of $\ito{n}$ the subset $\N{K}{n}$ of $\ito{n}$ by:
\[
\N{K}{n} \ldef \begin{cases} 
		\emptyset \quad &\text{if } K=\emptyset\\  
		\set{x\in \ito{n}\ohne K}{x > \min K  } \quad&\text{if } K\neq \emptyset
              \end{cases}
\]
\end{df}
Some examples are: $\N{\{1,2\}}{5}=\{3,4,5\} = \N{\{2\}}{5}, \N{\{1,3\}}{5}=\{2,4,5\}, \N{1,3,4,5}{5}=\{2\}$ and $\N{\{4,5\}}{5}=\emptyset$.
\begin{lem}\label{Dynkin lemma pJy(NJ)pJ annihilated by z}
For each subset $J$ of $\ito{n}$ the linear combination $\p{J} \yJ{\N{J}{n}} \p{J}$ of monomials in $\A{n}$ lies in the ideal $\Jk{|J|+1}{n}$ and is thus annihilated by $\z{|J|}{n}$ (see Corollary \ref{Dynkin Cor Ideal generators}).
\end{lem}
\begin{proof}
By Corollary \ref{Dynkin Cor Ideal generators} it suffices to show the identity:
\[
\y{|J|+1}{n}\p{J}\,\yJ{\N{J}{n}}\p{J} = \p{J}\,\yJ{\N{J}{n}}\p{J}
\]
This is a straightforward induction on $n$ requiring a case-by-case analysis on the cardinality of $J$ and considering the cases $n+1\in J$ and $n+1\notin J$ separately. One also needs the identity $\y{j}{n}\p{J}  = \p{J}$ (see Lemma  \ref{Idempotente Lemma Eigenschaften der y}).
\end{proof}
The next rather technical lemma, which we prove in more detail, is the heart of the proof of Theorem \ref{Dynkin theorem aJI neq 0}.
Recall that if $W_n$ as in Proposition \ref{Special cases  prop normal form of B_n lin.oriented}, then $\set{\p{w}}{w\in W_n}$ is a basis of $\A{n}$.
\begin{lem}\label{Idempotente Lemma mJK yNK hat nicht mJK}
Let $K\grn{n} J$. Then we have:
\[
\yJ{\N{K}{n}}\,\p{(K,J)}\in \biggl\langle \p{w}\in \A{n} \:\big|\:\: w \in W_n \text{ and }  \p{v}\neq \m{K,J}{n}\biggr\rangle_{R}\rdef \U{K,J}{n}
\]
\end{lem}
\begin{proof}
The proof is an induction on $n$. 

We start with a remark on the two extreme cases $\N{K}{n}=\emptyset$ and $J=K$ with $\N{K}{n}\neq \emptyset$: in the first case $\yJ{\N{K}{n}}=0$ holds, so the statement is clear. In the second case let $\yJ{\N{K}{n}}=\sum_{w\in W_n} c_w \p{w}$. Then we have $\m{K,K}{n}\,\yJ{\N{K}{n}} = \sum_{w\in W_n} c_w \p{K}\p{w}$. Now for all $w\in W_n$ with $c_w\neq 0$ (i.e.\ $\{w\}\subseteq \N{K}{n} \subset \ito{n}\ohne K$) the functors $\m{K,K}{n}=\p{K}$ and $\p{K}\p{w}$ differ in their action on the injective indecomposable $Q_n$-representations $I_i$ with $i\in \ito{n}\ohne K$. Hence $\m{K,K}{n}\,\yJ{\N{K}{n}}$ lies in $\U{K,K}{n}$. 

Since the calculations for $n\in \{1,2,3\}$ are trivial we proceed with the induction step $n\rightarrow n+1$ for $n\geq 3$. Let $K\grn{n+1}J$ such that $J\neq K$ and $\N{K}{n+1}\neq \emptyset$ hold.
\newline\case{$1^{\text{st}}$ case: ${n+1\notin K}$.} Since $\N{K}{n+1}=\N{K}{n}\cup \{n+1\}$ we have
 \[
 \yJ{\N{K}{n+1}}\m{K,J}{n+1} = \biggl(\yJ{\N{K}{n}}-\p{n+1}\yJ{\N{K}{n}}+\p{n+1}\biggr)\m{K,J}{n} = \yJ{\N{K}{n}}\m{K,J}{n}+\p{n+1}(1-\yJ{\N{K}{n}})\m{K,J}{n}
 \]
By the induction hypothesis the first summand $\yJ{\N{K}{n}}\m{K,J}{n}$ lies in $\U{K,J}{n}\subseteq \U{K,J}{n+1}$. Meanwhile each monomial appearing in the second summand starts with $\p{n+1}$, hence does not coincide with $\m{K,J}{n+1} = \m{K,J}{n} \in \A{n}$. (Compare the actions on $I_{n+1}$.) 
\newline\case{$2^{\text{nd}}$ case: ${n+1\in K}$} 
Then $N\ldef \N{K\ohne\{n+1\}}{n}=\N{K}{n+1} \neq \emptyset$ holds and we have:
\[
K\neq K_{\max}\qquad\text{and}\qquad \max N = \min K_{\max} -1 \neq 0
\]
We denote by $\widetilde{N}$ the set $ N\ohne \{\max N\}$. Let $\widetilde{J}$ and $\widetilde{K}$ be those subsets of $\ito{n}$ given by the definition of  $\m{K,J}{n+1}$ (depending on $n+1\in J$ or $n+1\neq J$) such that we have:
\[
\m{K,J}{n+1}=\p{ {K_{\max}}}\m{\widetilde{K},\widetilde{J}}{n}
\]
In the sequel we show: 
\begin{align*}
a &\ldef \p{ {K_{\max}}} \yJ{\widetilde{N}}    \m{\widetilde{K},\widetilde{J}}{n} \in \U{K,J}{n+1}\\[3pt]
b &\ldef \p{\max N}\p{ {K_{\max}}}\left(1- \yJ{\widetilde{N}}\right)  \m{\widetilde{K},\widetilde{J}}{n} \in \U{K,J}{n+1}
\end{align*}	
Then the claim follows immediately because we have
\[
\yJ{\N{K}{n+1}}\m{K,J}{n+1}= \left((1-\p{\max N})\yJ{\widetilde{N}}+\p{\max N}\right) \p{ {K_{\max}}}\m{\widetilde{K},\widetilde{J}}{n}=a+b 
\]
\newline\case{Proof of $b \in \U{K,J}{n+1}$.} Let $ c_w\p{\max N}\p{ {K_{\max}}}\p{w} \neq 0$ be a summand of $b$. Note that $\{w\}$ is a subset of $\ito{n}$. We look at the action on the injective indecomposable $Q_{n+1}$-representation $I_{n+1}$ to distinguish between $\p{\max N}\p{K_{\max}}\p{w}$ and $\m{K,J}{n+1}$:
\begin{align*}
\p{\max N}\p{ {K_{\max}}}\p{w}(I_{n+1}) &= \p{\max N}\p{ {K_{\max}}}(I_{n+1})
 = \p{\max N} (I_{\min K_{\max}-1})
 = I_{\min K_{\max}-2 }\\[3pt]
 &\neq I_{\min K_{\max}-1 } 
 = \p{ {K_{\max}}}(I_{n+1})
  = \p{ {K_{\max}}}\m{\widetilde{K},\widetilde{J}}{n}(I_{n+1})
 =\m{K,J}{n+1}(I_{n+1})
\end{align*}
\case{Proof of $a \in \U{K,J}{n+1}$.} For $\widetilde{N}= \emptyset$ the summand $a$ equals $0$. So we assume $\widetilde{N} \neq \emptyset$. 
Now we consider: 
\[
 s\ldef \begin{cases}
            n \quad &\text{if } n+1\notin J\\
            j\ldef \min J_{\max} -2 \quad &\text{if } n+1\in J\\
        \end{cases}
\]
In each case $\widetilde{J}$ and $\widetilde{K}$ both lie in $\A{s}$ and we have
\begin{align*}
\widetilde{N}= \N{\widetilde{K}}{s}  \qquad\text{and}\qquad  \m{\widetilde{K},\widetilde{J}}{n} =\m{\widetilde{K},\widetilde{J}}{s}\\
\shortintertext{In particular it follows:}
a = \p{ {K_{\max}}} \yJ{\widetilde{N}}\m{\widetilde{K},\widetilde{J}}{n}=\p{ {K_{\max}}} \yJ{\N{\widetilde{K}}{s}}\m{\widetilde{K},\widetilde{J}}{s}
\end{align*}
Now we consider an arbitrary summand $0\neq c_w\p{w} = c_w\p{K_{\max}}\p{v}$ of $a$. By the induction hypothesis $\p{v}$ and $\m{\widetilde{K},\widetilde{J}}{s}$ are distinct monomials in $\A{s}$. Recall that by the proof of Proposition \ref{Special cases  prop normal form of B_n lin.oriented} there thus exist an injective indecomposable $Q_s$-representation $I_j$ with $j\in \ito{s}$ and an index $x \in \oto{\smash{j}}$ with:
\[
 I_x=\p{v}(I_j) \neq \m{\widetilde{K},\widetilde{J}}{s}(I_j)
\]
Let $u\in W_s$ with $\m{\widetilde{K},\widetilde{J}}{s} = \p{\widetilde{K}_{\max}}\p{u}$ and let $y\in \oto{\smash{j}}$ be the index such that we have:
\[
 I_y=\p{u}(I_j)
\]
In particular we get:
\[
 \m{\widetilde{K},\widetilde{J}}{s}(I_j)= \p{\widetilde{K}_{\max}} I_y = \begin{cases}
                                                                          I_y \quad &\text{if } y\notin \widetilde{K}_{\max}\\
                                                                          I_{\min \widetilde{K}_{\max} -1} \quad &\text{if }  y\in \widetilde{K}_{\max}
                                                                         \end{cases}
\]
Since $ n+1 > y\notin \widetilde{K}_{\max}$ implies $y\notin K_{\max}$ we conclude $\m{K,J}{n+1}(I_j)= \m{\widetilde{K},\widetilde{J}}{s}(I_j)$ from:
\[
\m{K,J}{n+1}(I_j) = \p{K_{\max}}\p{{\widetilde{K}_{\max}}}\p{u}(I_j) =\begin{cases} 
                   \p{K_{\max}}I_y \quad &\text{if } y\notin \widetilde{K}_{\max}\\
                   \p{K_{\max}}I_{\min \widetilde{K}_{\max} -1} \quad &\text{if } y\in \widetilde{K}_{\max} 
                                                                         \end{cases}
\]
Meanwhile the action of $\p{w}$ on $I_j$ is:
\[
 \p{w}(I_j) = \p{K_{\max}}\p{v}(I_j) = \p{K_{\max}}(I_x) =\begin{cases}
		  I_x \quad &\text{if } x\notin {K}_{\max}\\
                  I_{\min K_{\max} -1} \quad &\text{if }  x\in K_{\max}
            \end{cases}
\]
We finish the proof with a case-by-case-comparison. 
\newline\case{$x\notin {K}_{\max}$:} Then $\p{w}(I_j)= \p{v}(I_j) \neq \m{\widetilde{K},\widetilde{J}}{s}(I_j) = \m{K,J}{n+1}(I_j)$.
\newline\case{$x\in {K}_{\max}$:} 
If $y\in \widetilde{K}_{\max}$ then $\p{w}$ and $\m{K,J}{n+1}$ act differently on $I_j$ since $\min \widetilde{K}_{\max} < \min K_{\max}$. So we assume $y\notin \widetilde{K}_{\max}$. We have to show, that $y\neq \min K_{\max} -1$. The only case in which $\min K_{\max}-1$ is not contained in $\widetilde{K}_{\max}$ is $n+1\in J$ and $J_{\max} = K_{\max}$. But in that case $y\leq j \leq s= \min J_{\max} -2= \min K_{\max}-2< \min K_{\max}-1 $ holds.
\end{proof}
\subsubsection{An element in $\mathbf{\f{K}{n}\A{n}\f{J}{n}}$, Proof of Theorem \ref{Dynkin theorem aJI neq 0}} 
There is also an alternative description of the idempotents $\g{J}{n}$. A straightforward induction on $n$ now shows: 
\begin{lem} \label{Dynkin lemma inductive gJn}
For each non-empty subset $J$ of $\ito{n}$ we have:
\[
\g{J}{n} = \p{J} -\yJ{\N{J}{n}}\p{J}
\]
In particular $\g{J}{n}$ is contained in the ideal $\Jk{k}{n}$ with $k=|J|$.
\end{lem}
\begin{proof}[Proof of Theorem \ref{Dynkin theorem aJI neq 0}.]
Let $K\grn{n} J$ and $k\ldef |K|=|J|$. With the previous lemma, Lemma \ref{Dynkin lemma pJy(NJ)pJ annihilated by z}, we gain the first reduction of the term/sum $\f{K}{n}\p{(K,J)}\f{J}{n}$:
\begin{align*}
\f{K}{n}\p{(K,J)}\f{J}{n} &=\quad \z{k}{n}\g{K}{n} \p{(K,J)} \g{J}{n} \:=\quad \z{k}{n}\g{K}{n}\p{(K,J)}\p{J}\biggl(\p{J}-\yJ{\N{J}{n}}\p{J}\biggr)\: = \quad \z{k}{n}\g{K}{n}\p{(K,J)}\\[3pt] 
\shortintertext{From the Lemma \ref{Idempotente Lemma Eigenschaften der y} and the chain description of $\p{(K,J)}$ we conclude next:}
\f{K}{n}\p{(K,J)}\f{J}{n} &=\quad\y{k}{n}\g{K}{n}\p{(K,J)} - \y{k+1}{n}\g{K}{n}\p{(K,J)} \:\in\quad \g{K}{n}\p{(K,J)} +\Jk{k+1}{n} 
\end{align*}
We finish the proof by looking closer at $\g{K}{n}\p{(K,J)}  =\p{(K,J)}- \yJ{\N{K}{n}}\p{(K,J)}$: on the one hand we have $\p{(K,J)} \notin \Jk{k+1}{n}$ as shown in Lemma \ref{Dynkin lemma pKJ not in Jk k+1}. On the other hand  
$\p{(K,J)}- \yJ{\N{K}{n}}\p{(K,J)}\neq 0$ holds by Lemma \ref{Idempotente Lemma mJK yNK hat nicht mJK}. Thus it follows $\f{K}{n}\p{(K,J)}\f{J}{n}\neq 0$.
\end{proof}
\subsection{Proof of the Main Theorem \ref{Dynkin maintheorem}}
By Theorem \ref{Dynkin complete pw orthogonal idempotents} and \ref{Dynkin theorem aJI neq 0}   $\set{\f{K}{n}\p{(K,J)}\f{J}{n}}{K\grn{n} J}$ is a linearly independent set with exactly $|\Po_n|$ elements. Thus bijectivity follows from Proposition \ref{Special cases  prop normal form of B_n lin.oriented}. To see the multiplicity let $M\grn{n} L$ and $K\grn{n} J$. If $K\neq L$ holds, then $\f{L}{n}$ and $\f{K}{n}$ are orthogonal. Hence $ \Phi(X_{(M,L)})\Phi(X_{(K,J)})=0 = \Phi(X_{(M,L)}X_{(K,J)}) = 0$. If $K=L$ we have the following identities, the fourth identity follows from Lemma \ref{Dynkin lemma pJy(NJ)pJ annihilated by z} and the fifth from the chain description):
\[
 \begin{aligned}
\Phi(X_{(M,L)})\Phi(X_{(K,J)}) &= \f{M}{n}\p{(M,K)}\f{K}{n}\f{K}{n}\p{(K,J)}\f{J}{n}\\
&= \f{M}{n}\p{(M,K)}\g{K}{n}\p{(K,J)}\f{J}{n}\\
&=\f{M}{n}\p{(M,K)}\biggl(\p{K}-\yJ{\N{K}{n}} \biggr)\p{K}\p{(K,J)}\f{J}{n}\\
&=\f{M}{n}\p{(M,K)}\p{K}\p{K}\p{(K,J)}\f{J}{n}\\
&=\f{M}{n}\p{(M,J)}\f{J}{n}\\
&= \Phi(X_{(M,L)}X_{(K,J)})
 \end{aligned}
\]
\qed

\providecommand{\bysame}{\leavevmode\hbox to3em{\hrulefill}\thinspace}
\providecommand{\MR}{\relax\ifhmode\unskip\space\fi MR }
\providecommand{\MRhref}[2]{%
  \href{http://www.ams.org/mathscinet-getitem?mr=#1}{#2}
}
\providecommand{\href}[2]{#2}

\end{document}